\documentclass[12pt,english,a4paper,oneside]{amsart}

\usepackage[cp1251]{inputenc}
\usepackage[english]{babel}
\usepackage{amsmath,amsthm,amssymb}

\usepackage[
a4paper,
left=22mm,
right=18mm,
top=20mm,
bottom=27mm
]{geometry}

\usepackage{xcolor}
\usepackage[colorlinks]{hyperref}

\newtheorem{theorem}{Theorem}[section]
\newtheorem{lemma}[theorem]{Lemma}

\numberwithin{equation}{section}

\theoremstyle{definition}

\newtheorem{remark}[theorem]{Remark}

\begin{document}

\title[Estimates of the total variation distance]{Estimates of the total variation distance between laws of Sobolev mappings on Gaussian spaces}

\author{Egor Kosov}

\address{\noindent Egor Kosov,
Centre de Recerca Matem\`atica, Campus de Bellaterra, Edifici~C 08193
Bellaterra (Barcelona), Spain.}
\email{kosoved09@gmail.com}

\author{Anastasiia Zhukova}

\address{\noindent
Anastasiia Zhukova,
Faculty of Mechanics and Mathematics, Lomonosov Moscow State University, Moscow, 119991 Russia}



\begin{abstract}
Under small-ball bounds for the Malliavin determinants of two
$\mathbb R^k$-valued Sobolev mappings on a Gaussian space, we estimate the
total variation distance between their laws both in terms of the
Kantorovich--Rubinstein distance and in terms of the distance between the
mappings in the corresponding Sobolev space.
In particular, our results yield new total variation distance estimates for
distributions of random vectors whose components belong to finite sums of
Wiener chaoses, with exponents improved by an asymptotic factor of two.

The proof is based on fractional regularity estimates for distributions of Sobolev mappings.
Namely, we show that if an $\mathbb R^k$-valued mapping has components
in $W^{2,p}(\gamma)$ and the determinant of the corresponding Malliavin matrix
satisfies a small-ball bound of order $\varkappa\in(0,1]$, then the law of the mapping has fractional regularity of order
\[
\frac{\varkappa}{1+(2k-1)\varkappa p^{-1}}.
\]
In particular, for large $p$, this gives regularity of order $\varkappa$ up to
an $O(p^{-1})$ loss.
\end{abstract}

\maketitle

\section{Introduction}

\subsection{Convergence of polynomial functionals on Gaussian spaces}

Let $\gamma$ be a centered Gaussian measure on a locally convex space $E$.
We denote by $\mathcal H_m(\gamma)$ the $m$th Wiener chaos associated with $\gamma$ and set
\[
\mathcal P_d(\gamma):=\bigoplus_{m=0}^d \mathcal H_m(\gamma).
\]
Thus, elements of $\mathcal P_d(\gamma)$ are polynomial functionals of the underlying random element, or equivalently finite sums of multiple stochastic integrals of orders at most $d$. Understanding the structural and asymptotic properties of their distributions is a classical problem in Gaussian analysis.

In the special case of normal approximation, the theory is by now very well developed. The fourth moment theorem of Nualart and Peccati~\cite{NP05} and its multidimensional extension by Peccati and Tudor~\cite{PT04} show that, within a fixed Wiener chaos, convergence to the normal law is governed by a remarkably simple moment condition. These results initiated a large body of work on quantitative and stronger forms of normal convergence on Wiener chaoses. See, for example,~\cite{ENP25,HMP24,HLN14,NP09,NP12,NN16} and the references therein.

In contrast, convergence to non-Gaussian limits within Wiener chaoses, and more generally for polynomial functionals of bounded degree, is substantially less understood. An important recent breakthrough in this direction is the result of Herry, Malicet and Poly~\cite{HMP}, which shows that the class of laws of elements of $\mathcal P_d(\gamma)$ is closed under weak convergence. 

The present paper addresses a complementary quantitative problem: to obtain
explicit estimates that upgrade convergence in distribution of random vectors
with sufficiently regular components, including components in
$\mathcal P_d(\gamma)$, to convergence in total variation. More precisely,
given two mappings
$f,g\colon E\to \mathbb R^k$,
we study estimates of the form
\begin{equation}\label{eq:intro-TV-KR-main}
d_{\rm TV}(f,g)\le C d_{\rm KR}(f,g)^\beta,
\end{equation}
where $d_{\rm TV}(f,g)$ denotes the total variation distance between the laws
of $f$ and $g$, while $d_{\rm KR}$ denotes the bounded
Kantorovich--Rubinstein distance, equivalently the Fortet--Mourier distance. 
The latter metrizes convergence in distribution. Here and below, when writing
distances between distributions, we do not distinguish notationally between a
mapping and its law under~$\gamma$.

\subsection{Known results and main contributions for polynomial mappings}

The first general result showing that convergence in distribution for finite sums of Wiener chaoses can be upgraded to convergence in total variation was obtained by Nourdin and Poly~\cite{NP13}. In quantitative form, they showed that if $f,g\in \mathcal P_d(\gamma)$ satisfy suitable two-sided variance assumptions, then \eqref{eq:intro-TV-KR-main} holds with exponent
$\beta=\frac{1}{2d+1}$.
Their result was subsequently strengthened in~\cite{Kos}, where the corresponding
exponent was improved to
$\beta=\frac{1}{d+1}$
(see also~\cite{Kos25}). This is close to optimal in the scalar-valued polynomial setting, since in that
case one cannot, in general, replace $\beta$ by any exponent larger than~$1/d$.

The vector-valued case is substantially more delicate. In particular, unlike
in the scalar-valued case, there is no clear candidate for the sharp exponent
$\beta$. To the best of our knowledge, all presently available obstruction
examples are scalar-valued.

The first general estimate of the form~\eqref{eq:intro-TV-KR-main} for random vectors of the form 
\[
f=(f_1,\ldots,f_k),\quad g=(g_1,\ldots,g_k),
\quad f_j,g_j\in\mathcal P_d(\gamma),
\]
was obtained by Nourdin, Nualart and Poly~\cite{NNP}. Under the assumptions
\begin{equation}\label{eq:intro-assumption}
	\int_E \Delta_f\,d\gamma\ge a,
	\quad
	\int_E \Delta_g\,d\gamma\ge a,
	\quad
	\max_{1\le j\le k}{\rm Var}_\gamma(f_j)\le b,
	\quad
	\max_{1\le j\le k}{\rm Var}_\gamma(g_j)\le b,
\end{equation}
where $\Delta_f$ and $\Delta_g$ denote the corresponding Malliavin determinants, they proved that~\eqref{eq:intro-TV-KR-main} holds for every
\[
\beta<
\frac{1}{(k+1)(4k(d-1)+3)+1}.
\]
This was subsequently improved in~\cite{BKZ} to the range
\[
\beta<
\frac{1}{4k(d-1)+1}.
\]

Our main result for polynomial mappings, Theorem~\ref{t2.2}, improves
this exponent by an asymptotic factor of two. Namely, under the same assumptions~\eqref{eq:intro-assumption}, the estimate~\eqref{eq:intro-TV-KR-main} holds for every
\[
\beta<\beta_1:=
\frac{1}{2k(d-1)+1}.
\]
The theorem also gives an endpoint-type estimate with exponent $\beta=\beta_1$, up to an additional logarithmic factor.

The improvement comes from total variation distance estimates in a more
general setting, namely for mappings with Sobolev components. For such
mappings, we obtain new bounds for the norm of the gradient of the
Malliavin determinant and for the norm of the adjugate of the Malliavin
matrix. Both new bounds contain an additional factor $\sqrt{\Delta_f}$, which
is responsible for the sharper total variation distance estimates.

\subsection{Mappings with Sobolev components}

Total variation distance estimates and related questions for mappings with
sufficiently regular components have been extensively studied in the Gaussian
setting and in more general frameworks. See, for example,
\cite{BC14,BC17,BC19,BCP,HMP25, Kos-FCAA,Kos-Adv,Kos23,KosZh}.

We prove general estimates of the form~\eqref{eq:intro-TV-KR-main}
for mappings with Sobolev components. More precisely, Theorem~\ref{t2.1} below asserts that, for mappings
\[
f=(f_1,\ldots,f_k),\quad g=(g_1,\ldots,g_k),
\quad f_j,g_j\in W^{2,p}(\gamma),
\]
the estimate~\eqref{eq:intro-TV-KR-main} holds with exponent
\[
\beta=
\frac{\varkappa}{1+\varkappa+\frac{\varkappa(2k-1)}{p}}
\]
provided that $p\ge 2k$, the Sobolev norms of the mappings are bounded, and
the corresponding small-ball estimates
\[
\gamma(\Delta_f\le \varepsilon)\le a\varepsilon^\varkappa,\quad
\gamma(\Delta_g\le \varepsilon)\le a\varepsilon^\varkappa,
\quad \forall \varepsilon>0,
\]
hold.
This improves the previously known result from~\cite{Kos-FCAA}, where, under
similar assumptions, estimate~\eqref{eq:intro-TV-KR-main} was established with
the smaller exponent
\[
\beta=\frac{\varkappa}{2+\varkappa+\frac{\varkappa(4k-1)}{p}}.
\]

This general Sobolev estimate applies to mappings with components in
$\mathcal P_d(\gamma)$. Indeed, such components belong to all Sobolev spaces
$W^{2,p}(\gamma)$, $p\in[1,\infty)$. Moreover, under the standard Malliavin-type
nondegeneracy assumptions
\[
\int_E \Delta_f\,d\gamma\ge a>0,
\quad
\int_E \Delta_g\,d\gamma\ge a>0,
\]
the required small-ball estimates follow from the Carbery--Wright inequality.
Since
\[
\Delta_f,\Delta_g\in\mathcal P_{2k(d-1)}(\gamma),
\]
we obtain
\[
\gamma(\Delta_f\le \varepsilon)
\le
C(a,d,k)\varepsilon^\varkappa,
\quad
\gamma(\Delta_g\le \varepsilon)
\le
C(a,d,k)\varepsilon^\varkappa,
\quad \varepsilon>0,
\]
with
\[
\varkappa=\frac{1}{2k(d-1)}.
\]
Substituting this value of $\varkappa$ into Theorem~\ref{t2.1} and using the
fact that polynomial components belong to $W^{2,p}(\gamma)$ for all
$p\in[1,\infty)$ gives the corresponding improved estimate of the
form~\eqref{eq:intro-TV-KR-main} for polynomial mappings.

\subsection{Regularity of distributions}

Estimates of the form~\eqref{eq:intro-TV-KR-main} are closely related to
fractional regularity of the corresponding laws. Indeed, applying
\eqref{eq:intro-TV-KR-main} to $g=f+h$, where $h\in\mathbb R^k$, gives
\[
d_{\rm TV}(f,f+h)\le C|h|^\beta.
\]
This translation estimate implies absolute continuity of the law of $f$ and,
more precisely, shows that its density $\varrho$ belongs to the
Nikolskii--Besov space $B^\beta_{1,\infty}(\mathbb R^k)$, that is,
\[
\int_{\mathbb R^k}|\varrho(x+h)-\varrho(x)|\,dx
\le C|h|^\beta,\quad \forall h\in\mathbb R^k,
\]
for some constant $C>0$ (see, for example,~\cite{BIN,Stein}).

Our approach is therefore based on first deriving Besov-type fractional
regularity estimates for the laws of mappings with Sobolev components, and
then using a smoothing argument to obtain total variation estimates of the
form~\eqref{eq:intro-TV-KR-main}.

A convenient way to verify such regularity, well suited to Malliavin-type
integration by parts on a Gaussian space, was introduced in~\cite{BKZ,Kos}
and further developed in~\cite{Kos-FCAA,Kos-MS}. For
a Borel measure $\mu$ on $\mathbb R^k$ and $t>0$, set
\[
\sigma(\mu,t):=\sup\Bigl\{\int_{\mathbb R^k}\partial_\theta \varphi\,d\mu
\colon \varphi\in C_0^\infty(\mathbb R^k),\ \|\varphi\|_\infty\le t,\
\|\partial_\theta\varphi\|_\infty\le 1,\ |\theta|=1\Bigr\}.
\]
For a measurable mapping $f\colon E\to\mathbb R^k$, we write
\[
\sigma_f(t):=\sigma(\gamma\circ f^{-1},t),
\]
where $\gamma\circ f^{-1}$ denotes the image measure of $\gamma$ under $f$.
This modulus is equivalent, up to dimensional constants, to the usual $L^1$
translation modulus of continuity of a density (see
\cite[Theorem 2.1]{Kos-MS} and \cite[Theorem 3.1]{Kos-FCAA}). In the present
notation, this can be written as
\[
2^{-1}\sup_{|h|\le t}d_{\rm TV}(f,f+h)
\le \sigma_f(t)
\le 6k \sup_{|h|\le t}d_{\rm TV}(f,f+h).
\]
The link between fractional regularity and estimates for the total variation
distance is provided by the smoothing inequality (see~\cite[Lemma 3.1]{Kos-FCAA}):
\begin{equation}\label{eq:intro-TV-est}
d_{\rm TV}(f,g)
\le 6\sqrt{k}\max\{\sigma_f(t),\sigma_g(t)\}
+\sqrt{k}\,t^{-1}d_{\rm KR}(f,g),
\quad
\forall t\in(0,1].
\end{equation}
After optimizing in $t$, this inequality turns estimates on $\sigma_f$ and
$\sigma_g$ into bounds of the form~\eqref{eq:intro-TV-KR-main}.

\subsection{Fractional regularity results}

Fractional regularity for densities of the laws of Sobolev mappings
$f=(f_1,\ldots,f_k)$, $f_j\in W^{2,p}(\gamma)$, was studied
in~\cite{Kos-FCAA}. It was proved there that, if the Sobolev norms of the
components are bounded and
\[
\gamma(\Delta_f\le \varepsilon)\le a\varepsilon^\varkappa,
\quad \forall \varepsilon>0,
\]
then
\[
\sigma_f(t)\le Ct^\alpha,\quad \forall t>0,
\]
with
\[
\alpha=
\frac{\varkappa}{2+\frac{(4k-1)\varkappa}{p}}.
\]
In the scalar-valued case, the same work gives the better exponent
\[
\alpha=\frac{\varkappa}{1+\frac{\varkappa}{p}}.
\]

Theorem~\ref{t1.1} below narrows this gap between the scalar-valued and
vector-valued cases. Under the same assumptions,
we prove the regularity estimate with the improved exponent
\[
\alpha=
\frac{\varkappa}{1+\frac{(2k-1)\varkappa}{p}}.
\]
For $k=1$, this recovers the scalar-valued exponent. Moreover, for large $p$,
the vector-valued regularity exponent improves from order $\varkappa/2$ to
order $\varkappa$.

In the case of polynomial mappings, Theorem~\ref{t1.2} asserts that for
mappings $f=(f_1,\ldots,f_k)$ with $f_j\in\mathcal P_d(\gamma)$ satisfying the
standard Malliavin-type nondegeneracy condition
\begin{equation}\label{eq:intro-nondegen-int}
\int_E \Delta_f\,d\gamma\ge a>0
\end{equation}
and suitable upper bounds on the variances of the components, one has
\[
\sigma_f(t)\le Ct^\alpha\quad \forall t>0
\]
for every $\alpha<\varkappa$, where
\[
\varkappa=\frac{1}{2k(d-1)}.
\]
The theorem also gives an endpoint-type regularity estimate with
$\alpha=\varkappa$, up to an additional logarithmic factor.

\subsection{Bounds in terms of Sobolev and $L^2$ distances}

We also obtain estimates for the total variation distance directly in terms of distances between the mappings themselves. Since the Kantorovich--Rubinstein distance is bounded by the $L^2(\gamma)$ distance between the mappings, estimates of the form~\eqref{eq:intro-TV-KR-main} immediately imply corresponding $L^2$-bounds. However, a direct argument often gives better exponents, and this leads to a second group of results in the paper.

In the scalar-valued case, estimates of this type go back to Davydov and Martynova~\cite{DM87} and were later developed by Nourdin and Poly~\cite{NP13}. A sharp one-dimensional estimate was obtained in~\cite{Kos-IMRN}: for every $f,g\in\mathcal P_d(\gamma)$,
\begin{equation}\label{eq:intro-TV-L2-1d}
	d_{\rm TV}(f,g)
	\le
	\frac{C(d)}{\operatorname{Var}_\gamma(f)^{1/(2d)}}
	\|f-g\|_{L^2(\gamma)}^{1/d}.
\end{equation}
Thus, in the scalar-valued case, the optimal power $1/d$ can be reached for
estimates in terms of the $L^2(\gamma)$ distance.

As in the estimates involving the Kantorovich--Rubinstein distance, we first prove a Sobolev version. In Theorem~\ref{t3.1}, for mappings
\[
f=(f_1,\ldots,f_k),\quad g=(g_1,\ldots,g_k),
\quad f_j,g_j\in W^{2,p}(\gamma),
\]
satisfying suitable Sobolev bounds and a small-ball assumption 
\[
\gamma(\Delta_f\le \varepsilon)\le a\varepsilon^\varkappa,
\quad \forall \varepsilon>0,
\]
on $\Delta_f$, we obtain an estimate of the form
\[
d_{\rm TV}(f,g)
\le
C\Bigl(\max_{1\le j\le k}\|f_j-g_j\|_{W^{1,p}(\gamma)}\Bigr)^\beta
\]
with
\[
\beta=\frac{\varkappa}{1+\frac{2k\varkappa}{p}}.
\]

For polynomial mappings $f,g$, this yields an improved vector-valued $L^2$ estimate. In~\cite{BKZ}, it was proved that, under the nondegeneracy condition~\eqref{eq:intro-nondegen-int} and an upper bound on the variances of the components of $f$, one has
\begin{equation}\label{eq:intro-TV-L2-main}
	d_{\rm TV}(f,g)
	\le
	C\Bigl(\max_{1\le j\le k}\|f_j-g_j\|_{L^2(\gamma)}\Bigr)^\beta
\end{equation}
for every
\[
\beta<\frac{1}{4k(d-1)}.
\]
Theorem~\ref{t3.2} improves this range by a factor of two. Namely, under the same assumptions on~$f$, the estimate~\eqref{eq:intro-TV-L2-main} holds for every
\[
\beta<\frac{1}{2k(d-1)}.
\]
As in the estimates in terms of the Kantorovich--Rubinstein distance, we also obtain an endpoint-type version with an additional logarithmic factor.

\subsection{Structure of the paper}

The rest of the paper is organized as follows. Section~\ref{sec:prelim} contains the necessary notation and preliminary facts on Gaussian Sobolev spaces, the Ornstein--Uhlenbeck operator, Malliavin matrices, and distances between probability laws. In Section~\ref{sec:key-lemmas}, we prove the key estimates for the Malliavin determinant and for the adjugate of the Malliavin matrix.

Section~\ref{sec:reg} is devoted to fractional regularity estimates for laws of mappings with Sobolev components and to their polynomial consequences. In Section~\ref{sec:KR}, these regularity estimates are used to prove the bounds in terms of Kantorovich--Rubinstein-type distances. Section~\ref{sec:L2} contains the estimates in terms of Sobolev and $L^2(\gamma)$ distances between the mappings.

\section{Preliminaries}
\label{sec:prelim}	

Let $C_0^\infty(\mathbb{R}^k)$ denote the space of all infinitely differentiable functions with compact support in $\mathbb{R}^k$, and let $C_b^\infty(\mathbb{R}^k)$ denote the space of all bounded infinitely differentiable functions on $\mathbb{R}^k$ whose derivatives of all orders are bounded. We denote by $\langle \cdot,\cdot\rangle$ the standard Euclidean inner product on $\mathbb{R}^k$, and by $|\cdot|$ the corresponding norm.

\subsection{Gaussian measures and the Cameron--Martin space}
	
Let $E$ be a real Hausdorff locally convex space with topological dual $E^*$. 
A Radon probability measure $\gamma$ on $E$ is called a centered Gaussian measure if, for every linear functional $\ell \in E^*$, either $\ell = 0$ $\gamma$-a.e., or the image measure $\gamma \circ \ell^{-1}$ is a centered Gaussian measure on $\mathbb{R}$, i.e., it has density
\[
\frac{1}{\sqrt{2\pi \sigma^2}} \exp\Bigl(-\frac{s^2}{2\sigma^2}\Bigr),
\quad \sigma \ge 0.
\]
As usual, $L^p(\gamma)$ denotes the space of functions whose $p$-th power is integrable with respect to the measure $\gamma$, and for
$f\in L^p(\gamma)$,
\[
\|f\|_{L^p(\gamma)} := \Bigl(\int_E |f|^p \, d\gamma\Bigr)^{1/p},
\quad p\in[1,\infty).
\]

For a centered Gaussian measure $\gamma$ on a locally convex space $E$, the corresponding Cameron--Martin space $H(\gamma)\subset E$ consists (see~\cite[\S 2.2, \S 2.4]{Gaus}) of all vectors $h\in E$ with finite Cameron--Martin norm:
\[
\|h\|_{H(\gamma)} := \sup \Bigl\{ \ell(h)\colon \ell \in E^*,\ \int_E \ell^2 \, d\gamma \le 1 \Bigr\}<\infty.
\]
For example, for the standard Gaussian measure $\gamma_n$ on $\mathbb{R}^n$ with density
\[
\frac{1}{(2\pi)^{n/2}} \exp\Bigl(-\frac{|x|^2}{2}\Bigr),
\]
the Cameron--Martin space coincides with $\mathbb{R}^n$ and the Cameron--Martin norm coincides with the usual Euclidean norm, that is, $H(\gamma_n)=\mathbb{R}^n$ and $\|h\|_{H(\gamma_n)} = |h|$.

Note that for any Radon Gaussian measure $\gamma$, its Cameron--Martin space $H(\gamma)$ is a separable Hilbert space (see~\cite[Theorem~3.2.7]{Gaus}). Let $\langle \cdot, \cdot \rangle_{H(\gamma)}$ denote the inner product in $H(\gamma)$.
Let $E_\gamma^*$ be the closure of $E^*$ in $L^2(\gamma)$. There is a natural duality between the spaces $E_\gamma^*$ and $H(\gamma)$ (see~\cite[\S 2.2, \S 2.4]{Gaus}). Namely, for every $\ell \in E_\gamma^*$ there exists an element $h_\ell \in H(\gamma)$ such that for all $\ell' \in E_\gamma^*$ one has
\[
\ell'(h_\ell)
= \ell(h_{\ell'})
= \langle h_\ell, h_{\ell'} \rangle_{H(\gamma)}
= \int_E \ell(x)\ell'(x)\, d\gamma(x).
\]
Conversely, for every $h \in H(\gamma)$ there exists $\ell_h \in E_\gamma^*$ such that $h = h_\ell.$

\subsection{Test functions and Sobolev spaces}
The space $\mathcal{FC}^\infty$ of cylindrical test functions consists of all functions $\varphi$ of the form
\begin{equation}\label{eq-test}
	\varphi(x) = \widetilde{\varphi}\bigl(\ell_1(x), \ldots, \ell_n(x)\bigr),
\end{equation}
where $n \in \mathbb{N}$, $\ell_1, \ldots, \ell_n \in E^*$, and $\widetilde{\varphi} \in C_b^\infty(\mathbb{R}^n)$.
For a function $\varphi \in \mathcal{FC}^\infty$ of this form,
its gradient along the Cameron--Martin space $H(\gamma)$ is defined by
\begin{equation}\label{eq-grad}
D^1_{H(\gamma)}\varphi(x) = \nabla \varphi(x)
:= \sum_{j=1}^n (\partial_j \widetilde{\varphi})\bigl(\ell_1(x), \ldots, \ell_n(x)\bigr)\, h_{\ell_j}.
\end{equation}
By the described duality, we have
\[
\langle \nabla \varphi(x), h \rangle_{H(\gamma)}
= \lim_{t \to 0} \frac{\varphi(x+th) - \varphi(x)}{t}
=: \partial_h \varphi(x),
\quad \forall\, h \in H(\gamma),
\]
and the definition of the gradient does not depend on the representation~\eqref{eq-test}.

Similarly, the second derivative $D^2_{H(\gamma)} \varphi(x)$ along $H(\gamma)$ is defined by
\begin{equation}\label{eq-2-deriv}
D^2_{H(\gamma)} \varphi(x)
= \sum_{j,k=1}^n (\partial_{j,k} \widetilde{\varphi})\bigl(\ell_1(x), \ldots, \ell_n(x)\bigr)\,
h_{\ell_j} \otimes h_{\ell_k},
\end{equation}
where $h_{\ell_j} \otimes h_{\ell_k} \in \mathcal{L}(H(\gamma))$ is the rank-one operator given by
\[
(h_{\ell_j} \otimes h_{\ell_k})h
= \langle h, h_{\ell_k} \rangle_{H(\gamma)}\, h_{\ell_j}.
\]
Equivalently, $D^2_{H(\gamma)} \varphi(x)$ is the operator on $H(\gamma)$ satisfying
\[
\langle D^2_{H(\gamma)} \varphi(x)\, h, q \rangle_{H(\gamma)}
= \partial_q \partial_h \varphi(x),
\quad \forall\, h, q \in H(\gamma).
\]	
	
The Sobolev space $W^{q,p}(\gamma)$, $q \in \{1,2\}$, is defined (see~\cite[\S 5.2]{Gaus}) as the closure of the space $\mathcal{FC}^\infty$ with respect to the norm
\[
\|\varphi\|_{W^{q,p}(\gamma)}
:= \|\varphi\|_{L^p(\gamma)} + \sum_{k=1}^q \|D^k_{H(\gamma)} \varphi\|_{L^p(\gamma)}.
\]
Here,
\[
\|D^1_{H(\gamma)} \varphi\|_{L^p(\gamma)}
= \|\nabla \varphi\|_{L^p(\gamma)}
:= \bigl\|\, \|\nabla \varphi\|_{H(\gamma)} \bigr\|_{L^p(\gamma)},
\]
and
\[
\|D^2_{H(\gamma)} \varphi\|_{L^p(\gamma)}
= \bigl\|\, \|D^2_{H(\gamma)} \varphi\|_{\mathrm{HS}} \bigr\|_{L^p(\gamma)},
\]
where $\|\cdot\|_{\mathrm{HS}}$ denotes the Hilbert--Schmidt norm of a linear operator on $H(\gamma)$.

In particular, for every $f \in W^{1,p}(\gamma)$, the Sobolev gradient $\nabla f \colon E \to H(\gamma)$ is well defined as the limit (in $L^p(\gamma; H(\gamma))$) of the gradients of functions in $\mathcal{FC}^\infty$ approximating $f$ in the Sobolev norm. Similarly, for $f \in W^{2,p}(\gamma)$, the second derivative $D^2_{H(\gamma)} f$ is well defined.

\subsection{The Ornstein--Uhlenbeck operator}
Let $L$ be the Ornstein--Uhlenbeck operator associated with the Gaussian
measure $\gamma$, that is, the generator of the Ornstein--Uhlenbeck semigroup
on $L^2(\gamma)$ (see~\cite[\S 1.4, \S 5.3]{Gaus}). For $p>1$, the operator
$L$ is well defined on $W^{2,p}(\gamma)$. In particular
(see~\cite[Theorem~5.7.1]{Gaus}),
\begin{equation}\label{eq3.1}
	\|Lf\|_{L^p(\gamma)} \le c(p)\|f\|_{W^{2,p}(\gamma)},
	\quad \forall\, f \in W^{2,p}(\gamma).
\end{equation}
We define an equivalent norm on $W^{2,p}(\gamma)$ by
\[
\|f\|_{\dot W^{2,p}(\gamma)}
:=\max\bigl\{\|f\|_{W^{2,p}(\gamma)},\, \|Lf\|_{L^p(\gamma)}\bigr\}.
\]
The operator $L$ is characterized by the following integration by parts formula
(see~\cite[Corollary~5.7.5]{Gaus}):
\begin{equation}\label{eq-int-by-parts}
\int_E \langle \nabla \varphi, \nabla \psi \rangle_{H(\gamma)}\, d\gamma
=
-\int_E \varphi\, L\psi\, d\gamma
=
-\int_E \psi\, L\varphi\, d\gamma
\quad \forall\, \varphi,\psi \in \mathcal{FC}^\infty.
\end{equation}

\subsection{The Malliavin matrix}	
For a mapping $f = (f_1,\ldots,f_k)\colon E \to \mathbb{R}^k$ with $f_j\in W^{1,1}(\gamma)$, the Malliavin matrix $M_f$ is defined by
\[
(M_f)_{i,j}(x) := \langle \nabla f_i(x), \nabla f_j(x) \rangle_{H(\gamma)},
\quad i,j = 1,\ldots,k.
\]
Let $A_f:=\operatorname{adj}(M_f)$ denote the adjugate matrix of $M_f$, i.e.,
\[
(A_f)_{i,j} = (M_f)^{j,i},
\]
where $(M_f)^{j,i}$ are the cofactors of $M_f$.
Let $\Delta_f:=\det M_f$. On the set $\{\Delta_f>0\}$, we have
\begin{equation}\label{inver}
\Delta_f\cdot M_f^{-1}=A_f.
\end{equation}
In the case $k=1$, one has
\[
M_f(x) = \Delta_f = |\nabla f(x)|_{H(\gamma)}^2
\quad\text{and}\quad A_f=1.
\]

\subsection{Distances between distributions}

For two $\gamma$-measurable mappings $f,g\colon E\to \mathbb{R}^k$,
we define the total variation distance between their distributions by
\[
d_{\mathrm{TV}}(f,g):=
\sup\Bigl\{
\int_E \bigl(\varphi(f)-\varphi(g)\bigr)\,d\gamma \colon
\varphi\in C_0^\infty(\mathbb{R}^k),\ \|\varphi\|_\infty\le 1
\Bigr\},
\]
where
\[
\|\varphi\|_\infty:=\sup_{x\in\mathbb{R}^k} |\varphi(x)|.
\]
The bounded Kantorovich--Rubinstein distance, also known as the
Fortet--Mourier distance, between these distributions is defined by
\[
d_{\mathrm{KR}}(f,g):=
\sup\Bigl\{
\int_E \bigl(\varphi(f)-\varphi(g)\bigr)\,d\gamma \colon
\varphi\in C_0^\infty(\mathbb{R}^k),\ \|\varphi\|_\infty\le 1,\ \|\nabla\varphi\|_\infty\le 1
\Bigr\}.
\]
We also consider the following Zolotarev-type distances, which generalize the Kantorovich--Rubinstein distance:
\[
d_r(f,g):=
\sup\Bigl\{
\int_E \bigl(\varphi(f)-\varphi(g)\bigr)\,d\gamma \colon
\varphi\in C_0^\infty(\mathbb{R}^k),\ \|\varphi\|_\infty\le 1,\
\max_{1\le j\le r}\max_{|\alpha|=j}\|\partial^\alpha \varphi\|_\infty\le 1
\Bigr\}.
\]
An important property of the metrics $d_r$ is that they satisfy a generalized relation of the form \eqref{eq:intro-TV-est}. Namely, by \cite[Lemma 2.3]{Kos23},
\begin{equation}\label{eq-TV-bound}
d_{\rm TV}(f,g)
\le
4\sqrt{k}\,\max\{\sigma_f(t),\sigma_g(t)\}
+
C_r(k)t^{-r}d_r(f,g),\quad \forall t\in(0,1],
\end{equation}
where $C_r(k)$ is a constant depending only on $r$ and $k$.

\subsection{Auxiliary functions}

For a nonnegative $\gamma$-measurable function $g\colon E\to \mathbb{R}$ and parameters $\varepsilon,r>0$, define
\[
h_{\gamma,r}(g,\varepsilon):=\int_{\varepsilon}^{\infty} s^{-r-1}\,\gamma(g\le s)\,ds.
\]
We note that
\begin{align}\label{eq1.1}
\int_E g^{-r}{\bf 1}_{\{g\ge \varepsilon\}}\,d\gamma
&=
r\int_E {\bf 1}_{\{g\ge \varepsilon\}} \int_{\varepsilon}^{\infty} {\bf 1}_{\{s\ge g\}} s^{-r-1}\,ds\,d\gamma
\\
&=
r\int_{\varepsilon}^{\infty} s^{-r-1}\gamma(\varepsilon\le g\le s)\,ds
\le r\,h_{\gamma,r}(g,\varepsilon).\nonumber
\end{align}

Fix a function $\Phi\in C^\infty(\mathbb{R})$ such that
\begin{equation}\label{eq-Phi}
\Phi(t)=0 \quad \text{for } t\in[-1,1], \quad
\Phi(t)=1 \quad \text{for } t\notin[-2,2], \quad
0\le \Phi(t)\le 1 \quad \text{for all } t\in\mathbb{R}.
\end{equation}
For $\varepsilon>0$, define $\Phi_\varepsilon(t):=\Phi(t/\varepsilon)$ for all $t\in\mathbb{R}$.

\begin{remark}\label{rem-Phi}
We can always choose such a function $\Phi$ so that $\|\Phi'\|_\infty\le 2$.
Indeed, consider a function $\eta\in C_0^\infty(\mathbb{R})$ such that
\[
\eta(t)=0 \ \text{for } t\in \mathbb{R}\setminus\bigl[-\tfrac12,\tfrac12\bigr],\quad
\eta(t)=1 \ \text{for } t\in \bigl[-\tfrac14,\tfrac14\bigr],\quad
0\le \eta(t)\le 1 \ \text{for all } t\in\mathbb{R}.
\]
Then
\[
I:=\int_{\mathbb{R}}\eta(t)\,dt\ge \int_{-\frac14}^{\frac14}\eta(t)\,dt=\frac12.
\]
Let $\eta^*:=I^{-1}\eta$ and set
\[
\Phi(t)=1-\int_{-\infty}^t \bigl(\eta^*(s+\tfrac32)-\eta^*(s-\tfrac32)\bigr)\,ds.
\]
Then $\Phi\in C^\infty(\mathbb{R})$ and satisfies
\[
\Phi(t)=0 \ \text{for } t\in[-1,1],\quad
\Phi(t)=1 \ \text{for } t\notin[-2,2],\quad
0\le \Phi(t)\le 1 \ \text{for all } t\in\mathbb{R}.
\]
Moreover,
\[
|\Phi'(t)|=
|\eta^*(t+\tfrac32)-\eta^*(t-\tfrac32)|
\]
and, since the supports of these two functions are disjoint,
\[
|\Phi'(t)|\le \|\eta^*\|_\infty = I^{-1}\|\eta\|_\infty \le 2.
\]
Thus,
\[
\|\Phi'\|_\infty\le 2.
\]
\end{remark}

\section{Key lemmas}
\label{sec:key-lemmas}

\begin{lemma}\label{lem1.1}
Let $J \colon \mathbb{R}^n \to \mathbb{R}^k$ be a linear mapping. 
Let $M = J J^*$, where $J^*$ denotes the adjoint operator, and let $A = \operatorname{adj}(M)$ be the adjugate matrix of $M$, i.e.
\[
MA = AM = (\det M)\, I.
\]
Then
\[
\|AJ\|_{\mathrm{op}} \le 2k^{\frac{1}{2}-\frac{k}{2}}(\det M)^{1/2}\,
(\operatorname{tr}M)^{\frac{k-1}{2}}
\]
and
\[
\|AJ\|_{\mathrm{HS}}\le
k^{1-\frac{k}{2}}(\det M)^{1/2}\,
(\operatorname{tr}M)^{\frac{k-1}{2}}.
\]
\end{lemma}

\begin{proof}
Since $M = JJ^*$ is symmetric, its adjugate $A$ is also symmetric, hence $A^* = A$. Therefore,
\[
\|AJ\|_{\mathrm{op}} = \|J^* A\|_{\mathrm{op}}.
\]
For any $h \in \mathbb{R}^k$, we compute
\[
|J^* A h|^2
= \langle J^* A h, J^* A h \rangle_{\mathbb{R}^n}
= \langle J J^* A h, A h \rangle_{\mathbb{R}^k}
= \langle M A h, A h \rangle_{\mathbb{R}^k}.
\]
Using $MA = (\det M) I$, we obtain
\[
|J^* A h|^2
= (\det M)\, \langle h, A h \rangle_{\mathbb{R}^k}.
\]
Taking the supremum over $|h|=1$, we get
\[
\|J^* A\|_{\mathrm{op}}^2 = (\det M)\, \|A\|_{\mathrm{op}}.
\]
	
For the Hilbert--Schmidt norm, we have
\[
\|AJ\|_{\rm HS}^2
=
\operatorname{tr}\bigl((AJ)(AJ)^*\bigr)
=
\operatorname{tr}(AJJ^*A^*)
=
\operatorname{tr}(AMA^*).
\]
Since $A^*=A$ and $MA=(\det M)I$, it follows that
\[
\|AJ\|_{\rm HS}^2
=
\operatorname{tr}(AMA)
=
(\det M)\operatorname{tr}A.
\]

Now, if $\lambda_1,\ldots,\lambda_k$ are the eigenvalues of $M$, then the eigenvalues of $A$ are
\[
\prod_{r\ne 1}\lambda_r,\ldots,\prod_{r\ne k}\lambda_r.
\]
Therefore,
\begin{equation}\label{eq-op-norm-est}
\|A\|_{\mathrm{op}} \le \Bigl(\frac{\operatorname{tr}M}{k-1}\Bigr)^{k-1}
\le e k^{1-k}(\operatorname{tr}M)^{k-1}.
\end{equation}
Moreover,
by Maclaurin's inequality,
\[
\operatorname{tr}A\le k^{2-k}(\operatorname{tr}M)^{k-1}.
\]
Taking square roots completes the proof.
\end{proof}

\begin{lemma}\label{lem1.2}
Let
$
f=(f_1,\ldots,f_k)\colon \mathbb R^n\to\mathbb R^k
$
be a mapping with $f_j\in C^\infty(\mathbb R^n)$ for $j=1,\ldots,k$.
Let
\[
J_f(x)=\bigl(\partial_i f_j(x)\bigr)_{\substack{1\le i\le n\\1\le j\le k}}
\]
be the Jacobian matrix of $f$, viewed as a $k\times n$ matrix, and let
$M_f=J_fJ_f^*$
be the corresponding Gram (Malliavin) matrix. 
Then, for $\Delta_f:=\det M_f$ and $A_f:=\operatorname{adj}(M_f)$, the following pointwise estimates hold:
\[
|\nabla \Delta_f|
\le
2k^{1-\frac{k}{2}}\,
\Delta_f^{1/2}
\Bigl(\sum_{j=1}^k |\nabla f_j|^2\Bigr)^{\frac{k-1}{2}}
\Bigl(\sum_{m=1}^k \|D^2 f_m\|_{\rm HS}^2\Bigr)^{1/2},
\]
and
\[
\sup_{|\theta|=1}
\Bigl|
\sum_{i,j=1}^k \theta_i
\langle \nabla f_j,\nabla (A_f)_{i,j}\rangle
\Bigr|
\le
5k^{\frac32-k}
\Bigl(\sum_{j=1}^k |\nabla f_j|^2\Bigr)^{k-1}
\Bigl(\sum_{m=1}^k \|D^2 f_m\|_{\rm HS}^2\Bigr)^{1/2}.
\]	
\end{lemma}

\begin{proof}
Since $M_f$ is symmetric and nonnegative definite, so is $A_f$, and
\begin{equation}\label{eq-adj-def}
A_f M_f = M_f A_f = \Delta_f I.
\end{equation}
Moreover,
\begin{equation}\label{eq-matr-der}
\partial_i M_f
=
\partial_i(J_fJ_f^*)
=
(\partial_i J_f)J_f^* + J_f(\partial_i J_f)^*.
\end{equation}
Hence
\begin{align}\label{eq-det-derivative}
\partial_i \Delta_f
=
\partial_i (\det M_f)
&=
\sum_{r,s=1}^k \frac{\partial (\det M_f)}{\partial (M_f)_{rs}}\,\partial_i (M_f)_{rs}
=
\sum_{r,s=1}^k (A_f)_{sr}\, (\partial_i M_f)_{rs}
\\
&=
\operatorname{tr}(A_f\,\partial_i M_f)
=
\operatorname{tr}\bigl(A_f(\partial_i J_f)J_f^*\bigr)
+
\operatorname{tr}\bigl(A_fJ_f(\partial_i J_f)^*\bigr).
\nonumber
\end{align}
Using cyclicity of the trace and the symmetry of $A_f$, both terms are equal:
\[
\operatorname{tr}\bigl(A_f(\partial_i J_f)J_f^*\bigr)
=
\operatorname{tr}\bigl(J_f^*A_f\,\partial_i J_f\bigr)
=
\operatorname{tr}\bigl((A_fJ_f)^*\partial_i J_f\bigr),
\]
and similarly
\[
\operatorname{tr}\bigl(A_fJ_f(\partial_i J_f)^*\bigr)
=
\operatorname{tr}\bigl((A_fJ_f(\partial_i J_f)^*)^*\bigr)
=
\operatorname{tr}\bigl((A_fJ_f)^*\partial_i J_f\bigr).
\]
Therefore,
\begin{equation}\label{eq-Delta-deriv}
\partial_i \Delta_f
=
2\,\operatorname{tr}\bigl((A_fJ_f)^*\partial_i J_f\bigr),
\end{equation}
and hence
\[
|\partial_i \Delta_f|
\le
2\,\|A_fJ_f\|_{\rm HS}\,\|\partial_i J_f\|_{\rm HS}.
\]
Summing over $i=1,\ldots,n$, we obtain
\[
|\nabla \Delta_f|^2
\le
4\,\|A_fJ_f\|_{\rm HS}^2
\sum_{i=1}^n \|\partial_i J_f\|_{\rm HS}^2.
\]
Now
\begin{equation}\label{eq-2der-norm-equality}
\sum_{i=1}^n \|\partial_i J_f\|_{\rm HS}^2
=
\sum_{r=1}^k \|D^2 f_r\|_{\rm HS}^2,
\end{equation}
and, by Lemma~\ref{lem1.1},
\begin{equation}\label{eq-det-grad-est}
|\nabla \Delta_f|^2
\le
4k^{2-k}(\det M_f)(\operatorname{tr}M_f)^{k-1}
\sum_{r=1}^k \|D^2 f_r\|_{\rm HS}^2.
\end{equation}

To obtain the second estimate, for a fixed $\theta\in\mathbb{R}^k$ with $|\theta|=1$, we note that
\[
\sum_{i,j=1}^k \theta_i
\langle \nabla f_j,\nabla (A_f)_{i,j}\rangle
=
\sum_{r=1}^n\sum_{j=1}^k \partial_r f_j
\sum_{i=1}^k \partial_r(A_f)_{i,j}\theta_i
=
\sum_{r=1}^n \langle \partial_r f,\partial_r A_f\theta\rangle.
\]
Now, by~\eqref{eq-adj-def}, for each $r=1,\ldots,n$, we have
\[
(\partial_r A_f) M_f + A_f(\partial_r M_f) = (\partial_r \Delta_f)\,I,
\]
and, multiplying from the right by $A_f$ and using~\eqref{eq-matr-der}, we obtain
\[
\Delta_f\,\partial_r A_f
=
(\partial_r\Delta_f)A_f - A_f(\partial_r M_f)A_f
=
(\partial_r\Delta_f)A_f
-
A_f(\partial_r J_f)J_f^*A_f
-
A_fJ_f(\partial_r J_f)^*A_f.
\]
Therefore,
\[
\Delta_f\sum_{r=1}^n\langle \partial_r f,\partial_r A_f\theta\rangle
=
\langle \nabla\Delta_f, J_f^*A_f\theta\rangle
-
\sum_{r=1}^n\langle (\partial_r J_f)^*A_f\,\partial_r f, J_f^*A_f\theta\rangle
-
\sum_{r=1}^n\langle J_f^*A_f\,\partial_r f, (\partial_r J_f)^*A_f\theta\rangle.
\]
We claim that
\[
\sum_{r=1}^n(\partial_r J_f)^*A_f\,\partial_r f
=
\frac12\,\nabla\Delta_f.
\]
Indeed, the $i$-th component of the left-hand side is
\[
\sum_{r=1}^n \langle \partial_r\partial_i f,\, A_f \partial_r f\rangle
=
\sum_{r=1}^n \langle \partial_i\partial_r f,\, A_f \partial_r f\rangle
=
\frac12\,\partial_i \Delta_f,
\]
by~\eqref{eq-Delta-deriv}. Hence
\[
\Delta_f\sum_{r=1}^n\langle \partial_r f,\partial_r A_f\theta\rangle
=
\frac{1}{2}\langle \nabla\Delta_f, J_f^*A_f\theta\rangle
-
\sum_{r=1}^n\langle J_f^*A_f\,\partial_r f, (\partial_r J_f)^*A_f\theta\rangle.
\]

For the first term, by Cauchy--Schwarz, Lemma~\ref{lem1.1}, and estimate~\eqref{eq-det-grad-est},
\[
|\langle \nabla\Delta_f, J_f^*A_f\theta\rangle|
\le
\|A_fJ_f\|_{\rm op}\,|\nabla\Delta_f|
\le
4k^{\frac32-k}\Delta_f
\Bigl(\sum_{j=1}^k |\nabla f_j|^2\Bigr)^{k-1}
\Bigl(\sum_{r=1}^k \|D^2f_r\|_{\rm HS}^2\Bigr)^{1/2}.
\]

For the second term, using Cauchy--Schwarz again,
\[
\Bigl|
\sum_{r=1}^n
\langle J_f^*A_f\,\partial_r f,(\partial_r J_f)^*A_f\theta\rangle
\Bigr|
\le
\Bigl(\sum_{r=1}^n |J_f^*A_f\,\partial_r f|^2\Bigr)^{1/2}
\Bigl(\sum_{r=1}^n |(\partial_r J_f)^*A_f\theta|^2\Bigr)^{1/2}.
\]
Now
\[
\sum_{r=1}^n |J_f^*A_f\,\partial_r f|^2
=\|J_f^*A_fJ_f\|_{\rm HS}^2
=
\Delta_f\,\operatorname{tr}(J_f^*A_fJ_f)
=
\Delta_f\,\operatorname{tr}(J_fJ_f^*A_f)
=
k\,\Delta_f^2,
\]
and, by~\eqref{eq-2der-norm-equality} and \eqref{eq-op-norm-est},
\[
\Bigl(\sum_{r=1}^n |(\partial_r J_f)^*A_f\theta|^2\Bigr)^{1/2}
\le
\|A_f\|_{\rm op}
\Bigl(\sum_{m=1}^k \|D^2f_m\|_{\rm HS}^2\Bigr)^{1/2}
\le
ek^{1-k}\Bigl(\sum_{j=1}^k |\nabla f_j|^2\Bigr)^{k-1}
\Bigl(\sum_{r=1}^k \|D^2f_r\|_{\rm HS}^2\Bigr)^{1/2}.
\]
Therefore,
\[
\Bigl|
\sum_{r=1}^n
\langle J_f^*A_f\,\partial_r f,(\partial_r J_f)^*A_f\theta\rangle
\Bigr|
\le
3k^{\frac32-k}\Delta_f
\Bigl(\sum_{j=1}^k |\nabla f_j|^2\Bigr)^{k-1}
\Bigl(\sum_{r=1}^k \|D^2f_r\|_{\rm HS}^2\Bigr)^{1/2}.
\]

Combining the two bounds and dividing by $\Delta_f$ on the set
$\{\Delta_f>0\}$, we arrive at
\[
\Bigl|
\sum_{i,j=1}^k \theta_i
\langle \nabla f_j,\nabla (A_f)_{i,j}\rangle
\Bigr|
\le
5k^{\frac32-k}
\Bigl(\sum_{j=1}^k |\nabla f_j|^2\Bigr)^{k-1}
\Bigl(\sum_{m=1}^k \|D^2f_m\|_{\rm HS}^2\Bigr)^{1/2},
\]
which implies the announced second estimate on the set $\{\Delta_f>0\}$.

It remains to remove the restriction $\Delta_f>0$. For $\varepsilon>0$, define
\[
f_\varepsilon(x,y):=f(x)+\varepsilon y,
\quad (x,y)\in\mathbb R^n\times\mathbb R^k .
\]
Then
\[
M_{f_\varepsilon}(x,y)=M_f(x)+\varepsilon^2 I,
\]
and hence $\det M_{f_\varepsilon}(x,y)>0$ for every
$(x,y)\in \mathbb R^n\times \mathbb R^k$. 
Moreover, $A_{f_\varepsilon}$ does not depend on $y$, and
\[
\nabla_x (A_{f_\varepsilon})_{i,j}\to \nabla_x (A_f)_{i,j}
\]
as $\varepsilon\downarrow0$. 
Applying the estimate already proved to $f_\varepsilon$ and letting
$\varepsilon\downarrow0$ gives the announced estimate for the mapping $f$ at
an arbitrary point $x$.
\end{proof}

\begin{lemma}\label{lem0}
Let $g_1,\ldots,g_k,h_1,\ldots,h_k\in L^p(\gamma)$, where $p>2k-1$. Then
\[
\Bigl\|
\Bigl(\sum_{j=1}^k |g_j|^2\Bigr)^{k-1}
\Bigl(\sum_{j=1}^k |h_j|^2\Bigr)^{1/2}
\Bigr\|_{L^{\frac{p}{2k-1}}(\gamma)}
\le
k^{k-\frac12}
\Bigl(
\max_{1\le j\le k}\max\{\|g_j\|_{L^p(\gamma)},\|h_j\|_{L^p(\gamma)}\}
\Bigr)^{2k-1}.
\]
\end{lemma}

\begin{proof}
The case $k=1$ is trivial. Assume $k\ge2$.
By H\"older's inequality,
\begin{align*}
\Bigl\|
\Bigl(\sum_{j=1}^k |g_j|^2\Bigr)^{k-1}
\Bigl(\sum_{j=1}^k |h_j|^2\Bigr)^{1/2}
\Bigr\|_{L^{\frac{p}{2k-1}}(\gamma)}
&\le
\Bigl\|\Bigl(\sum_{j=1}^k |g_j|^2\Bigr)^{1/2}\Bigr\|_{L^p(\gamma)}^{2k-2}
\cdot
\Bigl\|\Bigl(\sum_{j=1}^k |h_j|^2\Bigr)^{1/2}\Bigr\|_{L^p(\gamma)}
\\
&=
\Bigl\|\sum_{j=1}^k |g_j|^2\Bigr\|_{L^{p/2}(\gamma)}^{k-1}
\cdot
\Bigl\|\sum_{j=1}^k |h_j|^2\Bigr\|_{L^{p/2}(\gamma)}^{1/2}
\\
&\le
\Bigl(\sum_{j=1}^k \|g_j\|_{L^p(\gamma)}^2\Bigr)^{k-1}
\cdot
\Bigl(\sum_{j=1}^k \|h_j\|_{L^p(\gamma)}^2\Bigr)^{1/2}
\\
&\le
k^{k-\frac12}
\Bigl(
\max_{1\le j\le k}\max\{\|g_j\|_{L^p(\gamma)},\|h_j\|_{L^p(\gamma)}\}
\Bigr)^{2k-1}.
\end{align*}
This proves the claim.
\end{proof}

\begin{lemma}\label{lem:main}
Let $k\in\mathbb N$, let $p>2k-1$, and set
\[
q:=\frac{p}{p-2k+1}.
\]
Let $f=(f_1,\ldots,f_k)$ with $f_j\in W^{2,p}(\gamma)$ and
\[
\max_{1\le j\le k}\|f_j\|_{\dot{W}^{2,p}(\gamma)}\le b.
\]
Let $g\in W^{1,q}(\gamma)\cap L^\infty(\gamma)$, and let
$\theta=(\theta_1,\ldots,\theta_k)\in\mathbb R^k$ satisfy $|\theta|=1$.
Then for every $\varepsilon>0$,
\[
\biggl|
\int_E \sum_{i,j=1}^k \theta_i (A_f)_{i,j}
\langle \nabla f_j,\nabla g\rangle_{H(\gamma)}
\Delta_f^{-1}\Phi_\varepsilon(\Delta_f)\,d\gamma
\biggr|
\le
32\bigl(\|\Phi'\|_\infty+2\bigr)k\|g\|_\infty b^{2k-1}
\bigl(h_{\gamma,q}(\Delta_f,\varepsilon)\bigr)^{1/q}.
\]
\end{lemma}

\begin{proof}
\textbf{Step 1: the cylindrical case.}
Assume first that $g,f_1,\ldots,f_k\in \mathcal{FC}^\infty$ and that
\[
g(x)=\widetilde g\bigl(\ell_1(x),\ldots,\ell_n(x)\bigr),
\quad
f_j(x)=\widetilde f_j\bigl(\ell_1(x),\ldots,\ell_n(x)\bigr),
\quad j=1,\ldots,k,
\]
for some $\ell_1,\ldots,\ell_n\in E^*$.
Without loss of generality, we may assume that
$\{\ell_1,\ldots,\ell_n\}$ is an orthonormal system in $L^2(\gamma)$.
Then $\{h_{\ell_1},\ldots,h_{\ell_n}\}$ is an orthonormal system in $H(\gamma)$, and
the distribution of the mapping
\[
(\ell_1,\ldots,\ell_n)\colon (E,\gamma)\to\mathbb R^n
\]
is the standard Gaussian measure $\gamma_n$ on $\mathbb R^n$.
	
Set
\[
\Psi_\varepsilon(t):=t^{-1}\Phi_\varepsilon(t),
\]
where $\Phi_\varepsilon(t):=\Phi(t/\varepsilon)$ and $\Phi$ is the fixed function introduced in~\eqref{eq-Phi}.
Taking into account~\eqref{eq-grad} and applying the change-of-measure formula, we obtain
\[
\int_E \sum_{i,j=1}^k \theta_i (A_f)_{i,j}
\langle \nabla f_j,\nabla g\rangle_{H(\gamma)}
\Psi_\varepsilon(\Delta_f)\,d\gamma
=
\int_{\mathbb R^n}
\sum_{i,j=1}^k \theta_i (A_{\widetilde f})_{i,j}
\langle \nabla \widetilde f_j,\nabla \widetilde g\rangle
\Psi_\varepsilon(\Delta_{\widetilde f})\,d\gamma_n.
\]
Integrating by parts with respect to the Gaussian measure $\gamma_n$, we obtain
\[
-\int_{\mathbb R^n}\widetilde g \,(U_1+U_2)\,d\gamma_n,
\]
where
\[
U_1
=
\Psi_\varepsilon(\Delta_{\widetilde f})\,
\sum_{i,j=1}^k \theta_i
\bigl(
(A_{\widetilde f})_{i,j}L\widetilde f_j
+
\langle \nabla \widetilde f_j,\nabla (A_{\widetilde f})_{i,j}\rangle
\bigr),
\]
and
\[
U_2
=
\Psi_\varepsilon'(\Delta_{\widetilde f})\,
\sum_{i,j=1}^k \theta_i (A_{\widetilde f})_{i,j}
\langle \nabla \widetilde f_j,\nabla \Delta_{\widetilde f}\rangle.
\]
Hence
\[
\biggl|\int_E \sum_{i,j=1}^k \theta_i (A_f)_{i,j}
\langle \nabla f_j,\nabla g\rangle_{H(\gamma)}
\Psi_\varepsilon(\Delta_f)\,d\gamma\biggr|
\le
\|g\|_\infty\bigl(\|U_1\|_{L^1(\gamma_n)}+\|U_2\|_{L^1(\gamma_n)}\bigr).
\]
	
\medskip
\noindent
\textbf{Step 2: estimate of the $U_1$ term.}
To estimate $U_1$, we use \eqref{eq-op-norm-est}
and obtain
\[
\Bigl|\sum_{i,j=1}^k \theta_i
(A_{\widetilde f})_{i,j}L\widetilde f_j\Bigr|
\le 
\|A_{\widetilde{f}}\|_{\rm op}\Bigl(\sum_{j=1}^k|L\widetilde{f}_j|^2\Bigr)^{1/2}
\le 
ek^{1-k}\Bigl(\sum_{j=1}^k |\nabla \widetilde{f}_j|^2\Bigr)^{k-1}\Bigl(\sum_{j=1}^k|L\widetilde{f}_j|^2\Bigr)^{1/2}.
\]
Therefore, by Lemma \ref{lem0}, we obtain
\[
\Bigl\|\sum_{i,j=1}^k \theta_i
(A_{\widetilde f})_{i,j}L\widetilde f_j\Bigr\|_{L^{\frac{p}{2k-1}}(\gamma_n)}
\le 
ek^{1-k}
\Bigl\|
\Bigl(\sum_{j=1}^k |\nabla \widetilde{f}_j|^2\Bigr)^{k-1}\Bigl(\sum_{j=1}^k|L\widetilde{f}_j|^2\Bigr)^{1/2}
\Bigr\|_{L^{\frac{p}{2k-1}}(\gamma_n)}
\le 3kb^{2k-1}.
\]
From Lemma \ref{lem1.2} we have
\[
\Bigl|
\sum_{i,j=1}^k \theta_i
\langle \nabla \widetilde{f}_j,\nabla (A_{\widetilde{f}})_{i,j}\rangle
\Bigr|
\le
5k^{\frac32-k}
\Bigl(\sum_{j=1}^k |\nabla \widetilde{f}_j|^2\Bigr)^{k-1}
\Bigl(\sum_{m=1}^k \|D^2 \widetilde{f}_m\|_{\rm HS}^2\Bigr)^{1/2}
\]
and, applying Lemma \ref{lem0} once again,
we obtain
\[
\Bigl\|
\sum_{i,j=1}^k \theta_i
\langle \nabla \widetilde{f}_j,\nabla (A_{\widetilde{f}})_{i,j}\rangle
\Bigr\|_{L^{\frac{p}{2k-1}}(\gamma_n)}
\le 
5k^{\frac32-k}\cdot k^{k-\frac{1}{2}} b^{2k-1}
=5kb^{2k-1}.
\]
Since $\Phi_\varepsilon(t)=0$ for $|t|\le \varepsilon$ and $0\le \Phi_\varepsilon\le 1$,
\begin{equation}\label{eq-est-1}
\|\Psi_\varepsilon(\Delta_{\widetilde f})\|_{L^q(\gamma_n)}^q
=
\int_{\mathbb R^n}
\Delta_{\widetilde f}^{-q}\bigl(\Phi_\varepsilon(\Delta_{\widetilde f})\bigr)^q\,d\gamma_n
\le
\int_{\mathbb R^n}\Delta_{\widetilde f}^{-q}\mathbf 1_{\{\Delta_{\widetilde f}\ge \varepsilon\}}\,d\gamma_n
\le
q\,h_{\gamma_n,q}(\Delta_{\widetilde f},\varepsilon),
\end{equation}
by \eqref{eq1.1}. Hence, since $q^{1/q}\le e^{1/e}\le 2$, by H\"older's inequality, we obtain
\[
\|U_1\|_{L^1(\gamma_n)}
\le
16kb^{2k-1}\,
\bigl(h_{\gamma_n,q}(\Delta_{\widetilde f},\varepsilon)\bigr)^{1/q}.
\]
	
\medskip
\noindent
\textbf{Step 3: estimate of the $U_2$ term.}
Lemmas \ref{lem1.1} and \ref{lem1.2} imply
\[
\Bigl|\sum_{i,j=1}^k \theta_i (A_{\widetilde f})_{i,j}
\langle \nabla \widetilde f_j,\nabla \Delta_{\widetilde f}\rangle\Bigr|
\le
\|A_{\widetilde{f}}J_{\widetilde{f}}\|_{\rm op}
|\nabla \Delta_{\widetilde{f}}|\le 
4k^{\frac{3}{2}-k}\Delta_{\widetilde{f}}\,
\Bigl(\sum_{j=1}^k |\nabla \widetilde f_j|^2\Bigr)^{k-1}
\Bigl(\sum_{j=1}^k \|D^2 \widetilde{f}_j\|_{\rm HS}^2\Bigr)^{1/2}.
\]
Therefore
\[
|U_2|
\le
4k^{\frac{3}{2}-k}
\bigl|
\Phi_\varepsilon'(\Delta_{\widetilde f})
-
\Delta_{\widetilde f}^{-1}\Phi_\varepsilon(\Delta_{\widetilde f})
\bigr|
\Bigl(\sum_{j=1}^k |\nabla \widetilde f_j|^2\Bigr)^{k-1}
\Bigl(\sum_{j=1}^k \|D^2\widetilde f_j\|_{\rm HS}^2\Bigr)^{1/2}.
\]
Now, by Lemma \ref{lem0},
\[
\Bigl\|
\Bigl(\sum_{j=1}^k |\nabla \widetilde f_j|^2\Bigr)^{k-1}
\Bigl(\sum_{j=1}^k \|D^2\widetilde f_j\|_{\rm HS}^2\Bigr)^{1/2}
\Bigr\|_{L^{\frac{p}{2k-1}}(\gamma_n)}
\le
k^{k-\frac12}b^{2k-1}.
\]
Hence, by H\"older's inequality,
\[
\|U_2\|_{L^1(\gamma_n)}
\le
4kb^{2k-1}\,
\bigl\|
\Phi_\varepsilon'(\Delta_{\widetilde f})
-
\Delta_{\widetilde f}^{-1}\Phi_\varepsilon(\Delta_{\widetilde f})
\bigr\|_{L^q(\gamma_n)}.
\]
Since $|\Phi_\varepsilon'|\le \|\Phi'\|_\infty\varepsilon^{-1}I_{\{\varepsilon\le t\le 2\varepsilon\}}$,
we have
\[
\|\Phi_\varepsilon'(\Delta_{\widetilde f})\|_{L^q(\gamma_n)}
\le
2\|\Phi'\|_\infty\|\Delta_{\widetilde f}^{-1}I_{\{\Delta_{\widetilde f}\ge \varepsilon\}}\|_{L^q(\gamma_n)}
\le
4\|\Phi'\|_\infty\bigl(h_{\gamma_n,q}(\Delta_{\widetilde f},\varepsilon)\bigr)^{1/q},
\]
where we used that $q^{1/q}\le 2$.
Similarly, by \eqref{eq-est-1}, for 
$\Psi_\varepsilon(\Delta_{\widetilde f}):=\Delta_{\widetilde f}^{-1}\Phi_\varepsilon(\Delta_{\widetilde f})$
we have 
\[
\|\Delta_{\widetilde f}^{-1}\Phi_\varepsilon(\Delta_{\widetilde f})\|_{L^q(\gamma_n)}
\le 2\bigl(h_{\gamma_n,q}(\Delta_{\widetilde f},\varepsilon)\bigr)^{1/q}.
\]
Therefore,
\[
\|U_2\|_{L^1(\gamma_n)}
\le
8\bigl(2\|\Phi'\|_\infty+1\bigr)kb^{2k-1}\,
\bigl(h_{\gamma_n,q}(\Delta_{\widetilde f},\varepsilon)\bigr)^{1/q}.
\]
	
Combining the estimates for $U_1$ and $U_2$, we obtain
\begin{equation}\label{eq-est-2}
\biggl|\int_E \sum_{i,j=1}^k \theta_i (A_f)_{i,j}
\langle \nabla f_j,\nabla g\rangle_{H(\gamma)}
\Psi_\varepsilon(\Delta_f)\,d\gamma\biggr|
\le
8\bigl(2\|\Phi'\|_\infty+3\bigr)kb^{2k-1}\|g\|_\infty
\bigl(h_{\gamma_n,q}(\Delta_{\widetilde f},\varepsilon\bigr)^{1/q}.
\end{equation}
Since 
$\Delta_f = \Delta_{\widetilde f}\circ(\ell_1,\ldots,\ell_n)$, we have
\[
h_{\gamma_n,q}(\Delta_{\widetilde f},\varepsilon)
=
h_{\gamma,q}(\Delta_f,\varepsilon).
\]
Thus, the desired estimate holds in the cylindrical case.

\medskip
\noindent
\textbf{Step 4: approximation in $f$.}
Let now $f_j\in W^{2,p}(\gamma)$, and let $f_{j}^n\in \mathcal{FC}^\infty$ be such that
\[
f_{j}^n\to f_j \quad\text{in}\quad W^{2,p}(\gamma), \quad j=1,\ldots,k.
\]
Set $f^n=(f_{1}^n,\ldots,f_{k}^n)$.
It follows from the polynomial structure of the adjugate mapping and the convergence
$\nabla f_{j}^n\to \nabla f_j$ in $L^p(\gamma)$ that
\[
(A_{f^n})_{i,j}\to (A_f)_{i,j}
\quad\text{in}\quad L^{\frac{p}{2k-2}}(\gamma).
\]
Consequently,
\[
(A_{f^n})_{i,j}\langle \nabla f_{j}^n,\nabla g\rangle
\to 
(A_f)_{i,j}\langle \nabla f_j,\nabla g\rangle
\quad\text{in}\quad L^1(\gamma).
\]
Passing to a subsequence if necessary, we may assume that
\[
\nabla f_{j}^n\to \nabla f_j \quad \text{a.e.}
\]
and, consequently,
\[
\Delta_{f^n}\to \Delta_f \quad \text{a.e.}
\]
Therefore,
\[
\Psi_\varepsilon(\Delta_{f^n})
\to \Psi_\varepsilon(\Delta_f)
\quad \text{a.e.}
\]
Moreover,
\[
|\Psi_\varepsilon(\Delta_{f^n})|
\le \varepsilon^{-1}.
\]
Combining these facts, we conclude that
\[
\int_E \sum_{i,j=1}^k \theta_i (A_{f^n})_{i,j}
\langle \nabla f_{j}^n,\nabla g\rangle
\Psi_\varepsilon(\Delta_{f^n})\,d\gamma
\to
\int_E \sum_{i,j=1}^k \theta_i (A_f)_{i,j}
\langle \nabla f_j,\nabla g\rangle
\Psi_\varepsilon(\Delta_f)\,d\gamma.
\]

Since $\Delta_{f^n}\to \Delta_f$ $\gamma$-a.e., we have
\[
{\bf 1}_{\{\Delta_{f^n}(x)\le s\}}\to  {\bf 1}_{\{\Delta_f(x)\le s\}}
\]
for $(\gamma\otimes\lambda)$-a.e. $(x,s)\in E\times(\varepsilon,\infty)$, because for each fixed $x$ this convergence may fail only at the point $s=\Delta_f(x)$, which is negligible with respect to the Lebesgue measure. Hence, writing
\[
h_{\gamma,q}(\Delta_{f^n},\varepsilon)
=
\int_E\int_\varepsilon^\infty s^{-q-1} {\bf 1}_{\{\Delta_{f^n}(x)\le s\}}\,ds\, \gamma(dx),
\]
and using that
\[
0\le s^{-q-1}{\bf 1}_{\{\Delta_{f^n}(x)\le s\}}\le s^{-q-1},
\quad
\int_\varepsilon^\infty s^{-q-1}\,ds<\infty,
\]
the dominated convergence theorem yields
\[
h_{\gamma,q}(\Delta_{f^n},\varepsilon)\to h_{\gamma,q}(\Delta_f,\varepsilon).
\]
Thus, passing to the limit in \eqref{eq-est-2}
applied to $f^n_1, \ldots, f^n_k$ and $g\in \mathcal{FC}^\infty$, we obtain
\begin{equation}\label{eq-est-3}
\biggl|\int_E \sum_{i,j=1}^k \theta_i (A_f)_{i,j}
\langle \nabla f_j,\nabla g\rangle_{H(\gamma)}
\Psi_\varepsilon(\Delta_f)\,d\gamma\biggr|
\le
8\bigl(2\|\Phi'\|_\infty+3\bigr)kb^{2k-1}\|g\|_\infty
\bigl(h_{\gamma,q}(\Delta_{f},\varepsilon)\bigr)^{1/q}.
\end{equation}
for $g\in \mathcal{FC}^\infty$ and arbitrary $f_1, \ldots, f_k\in W^{2, p}(\gamma)$.

\medskip
\noindent
\textbf{Step 5: approximation in $g$.}
Let $g\in W^{1,q}(\gamma)\cap L^\infty(\gamma)$.
Choose $g^n\in \mathcal{FC}^\infty$ such that
\[
g^n\to g \quad\text{in}\quad W^{1,q}(\gamma)
\quad\text{and}\quad
g^n\to g \quad\text{a.e.}
\]
Let $\varphi\in C_0^\infty(\mathbb R)$ satisfy
\[
\varphi(t)=t \quad \text{for}\quad |t|\le \|g\|_{L^\infty(\gamma)}\quad\text{and}\quad
\|\varphi\|_\infty\le 2\|g\|_{L^\infty(\gamma)}.
\]
Then
\[
\varphi(g^n)\to g \quad\text{in}\quad W^{1,q}(\gamma)\quad\text{and}\quad
\|\varphi(g^n)\|_{L^\infty(\gamma)}\le 2\|g\|_{L^\infty(\gamma)}.
\]
Since
\[
\Psi_\varepsilon(\Delta_f)\sum_{i,j=1}^k \theta_i (A_f)_{i,j}\|\nabla f_j\|_{H(\gamma)}
\in L^{\frac{p}{2k-1}}(\gamma),
\]
and
$
\frac{1}{q}+\frac{2k-1}{p}=1,
$
H\"older's inequality yields
\[
\int_E \sum_{i,j=1}^k \theta_i (A_f)_{i,j}
\langle \nabla f_j,\nabla \varphi(g^n)\rangle_{H(\gamma)}
\Psi_\varepsilon(\Delta_f)\,d\gamma
\to
\int_E \sum_{i,j=1}^k \theta_i (A_f)_{i,j}
\langle \nabla f_j,\nabla g\rangle_{H(\gamma)}
\Psi_\varepsilon(\Delta_f)\,d\gamma.
\]
Applying \eqref{eq-est-3} to $\varphi(g^n)$ and passing to the limit, we obtain
\[
\biggl|
\int_E \sum_{i,j=1}^k \theta_i (A_f)_{i,j}
\langle \nabla f_j,\nabla g\rangle_{H(\gamma)}
\Psi_\varepsilon(\Delta_f)\,d\gamma
\biggr|
\le
16(2\|\Phi'\|_\infty+3)k\,\|g\|_\infty b^{2k-1}
\bigl(h_{\gamma,q}(\Delta_f,\varepsilon)\bigr)^{1/q},
\]
which proves the lemma.
\end{proof}

\section{Regularity of distributions}
\label{sec:reg}

\begin{theorem}\label{t1.1}
Let $k\in \mathbb{N}$, $p\ge 2k$, $b>0$, and set
\[
q:=\frac{p}{p-2k+1}.
\]
Then, for every mapping $f=(f_1,\ldots,f_k)$ with $f_j\in W^{2,p}(\gamma)$ and
\[
\max_{1\le j\le k}\|f_j\|_{\dot W^{2,p}(\gamma)}\le b,
\]
one has
\[
\sigma_f(t)\le
128k b^{2k-1}\, t\, \bigl(h_{\gamma,q}(\Delta_f,\varepsilon)\bigr)^{1/q}
+\gamma(\Delta_f\le 2\varepsilon)
\quad \forall\, t,\varepsilon>0.
\]

If, in addition,
\[
\gamma(\Delta_f\le \varepsilon)\le a\varepsilon^\varkappa
\quad \forall\, \varepsilon>0
\]
for some $\varkappa\in(0,q)$, then
\[
\sigma_f(t)\le
512 k^\alpha (1-\alpha)^{-\alpha} a^{\alpha/\varkappa} b^{(2k-1)\alpha}\cdot t^\alpha
\quad \forall\, t>0,
\]
where
\[
\alpha:=\frac{p\varkappa}{p+\varkappa(2k-1)}.
\]
\end{theorem}

\begin{proof}
Fix $t>0$, $\varepsilon>0$, a unit vector $\theta\in\mathbb R^k$, and
$\varphi\in C_0^\infty(\mathbb R^k)$ such that
\[
\|\varphi\|_\infty\le t,
\quad
\|\partial_\theta\varphi\|_\infty\le 1.
\]
Set
\[
g:=\varphi\circ f.
\]
Then, since $q\le p$, we have $g\in W^{1,q}(\gamma)\cap L^\infty(\gamma)$ and
$\|g\|_\infty\le t$.
By the chain rule,
\[
\nabla g=\sum_{m=1}^k (\partial_m\varphi)(f)\,\nabla f_m,
\]
hence
\[
\langle \nabla f_j,\nabla g\rangle_{H(\gamma)}
=
\sum_{m=1}^k (\partial_m\varphi)(f)\,
\langle \nabla f_j,\nabla f_m\rangle_{H(\gamma)}.
\]
Since $A_f$ is the adjugate of the Malliavin matrix
\[
M_f=\bigl(\langle \nabla f_i,\nabla f_j\rangle_{H(\gamma)}\bigr)_{i,j=1}^k,
\]
we have $A_fM_f=\Delta_f I$, and therefore
\[
\sum_{j=1}^k (A_f)_{i,j}\langle \nabla f_j,\nabla g\rangle_{H(\gamma)}
=
\Delta_f\,(\partial_i\varphi)(f).
\]
Multiplying by $\theta_i$ and summing over $i$, we obtain
\[
\sum_{i,j=1}^k \theta_i (A_f)_{i,j}
\langle \nabla f_j,\nabla g\rangle_{H(\gamma)}
=
\Delta_f\,(\partial_\theta\varphi)(f).
\]
Hence
\begin{align*}
\int_E \partial_\theta\varphi(f)\,d\gamma
&=
\int_E (\partial_\theta\varphi)(f)\,\Delta_f\,\Delta_f^{-1}\Phi_\varepsilon(\Delta_f)\,d\gamma
+\int_E (\partial_\theta\varphi)(f)\bigl(1-\Phi_\varepsilon(\Delta_f)\bigr)\,d\gamma \\
&=
\int_E
\sum_{i,j=1}^k \theta_i (A_f)_{i,j}
\langle \nabla f_j,\nabla g\rangle_{H(\gamma)}
\Delta_f^{-1}\Phi_\varepsilon(\Delta_f)\,d\gamma 
+\int_E (\partial_\theta\varphi)(f)\bigl(1-\Phi_\varepsilon(\Delta_f)\bigr)\,d\gamma.
\end{align*}
Taking absolute values and applying Lemma~\ref{lem:main}, together with the choice
of $\Phi$ such that $\|\Phi'\|_\infty\le 2$ from Remark~\ref{rem-Phi}, we get
\begin{align*}
\biggl|\int_E \partial_\theta\varphi(f)\,d\gamma\biggr|
&\le
32(\|\Phi'\|_\infty+2)k\|g\|_\infty b^{2k-1}
\,h_{\gamma,q}(\Delta_f,\varepsilon)^{1/q}
+\int_E \bigl(1-\Phi_\varepsilon(\Delta_f)\bigr)\,d\gamma \\
&\le
128k b^{2k-1}\, t\, \bigl(h_{\gamma,q}(\Delta_f,\varepsilon)\bigr)^{1/q}
+\gamma(\Delta_f\le 2\varepsilon).
\end{align*}
Taking the supremum over all admissible $\theta$ and $\varphi$, we obtain
\[
\sigma_f(t)\le
128k b^{2k-1}\, t\, \bigl(h_{\gamma,q}(\Delta_f,\varepsilon)\bigr)^{1/q}
+\gamma(\Delta_f\le 2\varepsilon),
\]
which proves the first assertion.
	
Assume now that
\[
\gamma(\Delta_f\le s)\le as^\varkappa
\quad \forall\, s>0
\]
for some $\varkappa\in(0,q)$. Then
\[
h_{\gamma,q}(\Delta_f,\varepsilon)
=
\int_\varepsilon^\infty s^{-q-1}\gamma(\Delta_f\le s)\,ds
\le
a\int_\varepsilon^\infty s^{-q-1+\varkappa}\,ds
=
\frac{a}{q-\varkappa}\,\varepsilon^{-q+\varkappa}.
\]
Therefore,
\[
\bigl(h_{\gamma,q}(\Delta_f,\varepsilon)\bigr)^{1/q}
\le
\Bigl(\frac{a}{q-\varkappa}\Bigr)^{1/q}\varepsilon^{-1+\varkappa/q},
\]
and also
\[
\gamma(\Delta_f\le 2\varepsilon)\le 2^\varkappa a\,\varepsilon^\varkappa.
\]
Hence
\[
\sigma_f(t)\le At\varepsilon^{-\beta}+B\varepsilon^\varkappa,
\]
where
\[
A:=128k b^{2k-1}\Bigl(\frac{a}{q-\varkappa}\Bigr)^{1/q},
\quad
B:=2^\varkappa a,
\quad
\beta:=1-\frac{\varkappa}{q}=\frac{q-\varkappa}{q}.
\]

For fixed $t>0$, the function
\[
\varepsilon\mapsto At\varepsilon^{-\beta}+B\varepsilon^\varkappa
\]
is minimized at
\[
\varepsilon_*=
\Bigl(\frac{\beta A t}{\varkappa B}\Bigr)^{\frac{1}{\beta+\varkappa}}.
\]
Substituting this value yields
\[
\sigma_f(t)\le
(\beta+\varkappa)\beta^{-\frac{\beta}{\beta+\varkappa}}
\varkappa^{-\frac{\varkappa}{\beta+\varkappa}}
A^{\frac{\varkappa}{\beta+\varkappa}}
B^{\frac{\beta}{\beta+\varkappa}}
t^{\frac{\varkappa}{\beta+\varkappa}}.
\]
Finally,
\[
\frac{\varkappa}{\beta+\varkappa}
=
\frac{\varkappa}{1+\varkappa-\varkappa/q}
=
\frac{p\varkappa}{p+\varkappa(2k-1)}
=
\alpha
\]
and, since $\alpha^{-\alpha}(1-\alpha)^{-(1-\alpha)}\le 2$,
we obtain
\begin{align*}
(\beta+\varkappa)\beta^{-\frac{\beta}{\beta+\varkappa}}
\varkappa^{-\frac{\varkappa}{\beta+\varkappa}}
A^{\frac{\varkappa}{\beta+\varkappa}}
B^{\frac{\beta}{\beta+\varkappa}}
&=
\alpha^{-\alpha}(1-\alpha)^{-(1-\alpha)}
128^\alpha 2^{\alpha\beta}
k^\alpha (q-\varkappa)^{-\alpha/q} a^{\alpha/\varkappa} b^{(2k-1)\alpha} 
\\
&\le
512\, k^\alpha (q-\varkappa)^{-\alpha/q} a^{\alpha/\varkappa} b^{(2k-1)\alpha}
\\
&\le 
512\, k^\alpha (1-\alpha)^{-\alpha} a^{\alpha/\varkappa} b^{(2k-1)\alpha},
\end{align*}
where in the last step we used the estimate
\[
(q-\varkappa)^{-\alpha/q}\le 
(1-\alpha)^{-\alpha/q}\le (1-\alpha)^{-\alpha},
\]
which follows from
\[
1-\alpha
=
1-\frac{\varkappa}{1+\varkappa-\varkappa/q}
=
\frac{q-\varkappa}{q-\varkappa+\varkappa q}
\le q-\varkappa.
\]
This proves the second assertion.
\end{proof}

\begin{theorem}\label{t1.2}
Let $k,d\in\mathbb{N}$ with $d\ge 2$, and let $a,b>0$. Set
\[
\varkappa:=\frac{1}{2k(d-1)}.
\]
Then, for every mapping $f=(f_1,\ldots,f_k)$ with $f_j\in \mathcal{P}_d(\gamma)$, satisfying
\[
\int_E \Delta_f\,d\gamma \ge a
\quad\text{and}\quad
\max_{1\le j\le k}\operatorname{Var}_\gamma(f_j)\le b,
\]
and for every $\alpha\in(0,\varkappa)$, one has
\[
\sigma_f(t)
\le
C_1kd(\varkappa-\alpha)^{-\frac{d(2k-1)}{4k(d-1)}}
\bigl(a^{-1}b^{\frac{2k-1}{2}}t\bigr)^\alpha
\quad \forall t>0,
\]
where $C_1>0$ is an absolute constant. In particular, optimizing over $\alpha$, we obtain
\[
\sigma_f(t)
\le
C_2kd
\bigl(1+\bigl|\ln\bigl(a^{-1}b^{\frac{2k-1}{2}}t\bigr)\bigr|\bigr)^{\frac{d(2k-1)}{4k(d-1)}}
\bigl(a^{-1}b^{\frac{2k-1}{2}}t\bigr)^\varkappa
\quad \forall t>0,
\]
for some absolute constant $C_2>0$.
\end{theorem}

\begin{proof}
Replacing each component $f_j$ by
\[
f_j-\int_E f_j\,d\gamma
\]
does not change either $\Delta_f$ or $\sigma_f$. Therefore, without loss of generality, we may assume that
\[
\int_E f_j\,d\gamma=0
\quad \forall\, j=1,\ldots,k.
\]
Hence
\[
\|f_j\|_{L^2(\gamma)}^2={\rm Var}_\gamma(f_j)\le b
\quad \forall\, j=1,\ldots,k.
\]

We now fix some $p\ge 2k\ge2$, which will be specified later.
We note that for every $g\in \mathcal P_d(\gamma)$, the standard hypercontractivity property of the Ornstein--Uhlenbeck semigroup implies
(see, for instance, \cite[Theorem~5.5.3]{Gaus} or \cite[Theorem~5.10 and Remark~5.11]{Jan97})
that
\begin{equation}\label{eq-hyper}
	\|g\|_{L^p(\gamma)}\le (p-1)^{d/2}\|g\|_{L^2(\gamma)}
	\quad \forall\, g\in \mathcal P_d(\gamma).
\end{equation}
For a fixed orthonormal basis $\{e_i\}_{i=1}^\infty$ in $H(\gamma)$, since differentiation lowers the degree of a polynomial by one, we have
\[
\partial_{e_i} f_j\in \mathcal P_{d-1}(\gamma),
\quad
\partial_{e_i}\partial_{e_r} f_j\in \mathcal P_{d-2}(\gamma).
\]
Moreover, for every $g\in \mathcal H_m(\gamma)$ one has
\[
\|\nabla g\|_{L^2(\gamma)}^2
=
\int_E \langle \nabla g,\nabla g\rangle_{H(\gamma)}\,d\gamma
=
-\int_E gLg\,d\gamma
=
m\|g\|_{L^2(\gamma)}^2,
\]
since elements of $\mathcal H_m(\gamma)$ are eigenfunctions of the
Ornstein--Uhlenbeck operator $L$ with eigenvalue $-m$.
Consequently,
\[
\|D_H^2g\|_{L^2(\gamma)}^2
=
m(m-1)\|g\|_{L^2(\gamma)}^2.
\]
Therefore, expanding $f_j$ into the sum of its orthogonal chaos components, we obtain
\[
\|\nabla f_j\|_{L^2(\gamma)}
\le
\sqrt{d}\,\|f_j\|_{L^2(\gamma)}
\]
and
\[
\|D_H^2 f_j\|_{L^2(\gamma)}
\le
\sqrt{d(d-1)}\,\|f_j\|_{L^2(\gamma)}
\le d\,\|f_j\|_{L^2(\gamma)}.
\]
We now combine these bounds with \eqref{eq-hyper}. Since $p/2\ge 1$, we have
\begin{align*}
\|D_H^2 f_j\|_{L^p(\gamma)}^2
&=
\Bigl\|\sum_{i,r=1}^\infty |\partial_{e_i}\partial_{e_r}f_j|^2\Bigr\|_{L^{p/2}(\gamma)}
\le
\sum_{i,r=1}^\infty \||\partial_{e_i}\partial_{e_r}f_j|^2\|_{L^{p/2}(\gamma)}
\\
&=
\sum_{i,r=1}^\infty \|\partial_{e_i}\partial_{e_r}f_j\|_{L^p(\gamma)}^2
\le
(p-1)^{d-2}
\sum_{i,r=1}^\infty \|\partial_{e_i}\partial_{e_r}f_j\|_{L^2(\gamma)}^2
\\
&=
(p-1)^{d-2}\|D_H^2 f_j\|_{L^2(\gamma)}^2
\le
d^2(p-1)^{d-2}\|f_j\|_{L^2(\gamma)}^2.
\end{align*}
Thus
\begin{equation}\label{eq-D-2}
\|D_H^2 f_j\|_{L^p(\gamma)}
\le d(p-1)^{(d-2)/2}\|f_j\|_{L^2(\gamma)}.
\end{equation}
Similarly,
\begin{equation}\label{eq-D-1}
\|\nabla f_j\|_{L^p(\gamma)}
\le \sqrt{d}(p-1)^{(d-1)/2}\|f_j\|_{L^2(\gamma)}.
\end{equation}
Furthermore, since elements of $\mathcal H_m(\gamma)$ are eigenfunctions of the Ornstein--Uhlenbeck operator $L$ with eigenvalue $-m$, we have
\[
\|Lf_j\|_{L^2(\gamma)}\le d\|f_j\|_{L^2(\gamma)}.
\]
Applying \eqref{eq-hyper} again, we obtain
\begin{equation}\label{eq-L}
\|Lf_j\|_{L^p(\gamma)}
\le
(p-1)^{d/2}\|Lf_j\|_{L^2(\gamma)}
\le
d(p-1)^{d/2}\|f_j\|_{L^2(\gamma)}.
\end{equation}
Therefore, 
\begin{equation}\label{eq-Sobolev-est}
\|f_j\|_{\dot W^{2,p}(\gamma)}
\le 3d (p-1)^{d/2}\|f_j\|_{L^2(\gamma)}\le
3d (p-1)^{d/2}\sqrt b.
\end{equation}

Now, since 
$\Delta_f\in \mathcal{P}_{2k(d-1)}(\gamma)$,
the Carbery--Wright inequality (see \cite{CW01}) implies that
\begin{equation}\label{eq:CW}
\gamma(\Delta_f\le \varepsilon)\le Ck(d-1)\, a^{-\frac{1}{2k(d-1)}}\varepsilon^{\frac{1}{2k(d-1)}}
\quad \forall\, \varepsilon>0,
\end{equation}
for some absolute constant $C\ge 1$.
Applying Theorem~\ref{t1.1}, we obtain
\[
\sigma_f(t)\le
512 k^{\alpha_p} (1-\alpha_p)^{-\alpha_p}
\bigl( Ck(d-1)\, a^{-\frac{1}{2k(d-1)}}\bigr)^{\alpha_p/\varkappa}
\bigl(3d (p-1)^{d/2}\sqrt b\bigr)^{(2k-1)\alpha_p}\cdot t^{\alpha_p}
\quad \forall\, t>0,
\]
where
\[
\varkappa:=\frac{1}{2k(d-1)}
\quad\text{and}\quad
\alpha_p:=\frac{p\varkappa}{p+\varkappa(2k-1)}
=
\frac{1}{2k(d-1)+(2k-1)p^{-1}}.
\]
Since $\alpha_p\le \varkappa\le 1/2$, we have
\[
(1-\alpha_p)^{-\alpha_p}\le \sqrt2.
\]
Moreover,
\[
\bigl( Ck(d-1)\, a^{-\frac{1}{2k(d-1)}}\bigr)^{\alpha_p/\varkappa}
=
\bigl(Ck(d-1)\bigr)^{\alpha_p/\varkappa}a^{-\alpha_p}
\le Ckda^{-\alpha_p},
\]
because $\alpha_p/\varkappa\le 1$ and $Ck(d-1)\ge 1$.
Also,
\[
(2k-1)\alpha_p\le (2k-1)\varkappa\le\frac{1}{d-1}\le \frac2d
\quad\text{and}\quad
\alpha_p\le\varkappa\le \frac{1}{2k},
\]
which imply that
\[
(3d)^{(2k-1)\alpha_p}\le (3d)^{2/d}\le 12,
\quad
(p-1)^{\frac{d(2k-1)}{2}\alpha_p}\le p^{\frac{d(2k-1)}{2}\varkappa},
\quad
k^{\alpha_p}\le 2.
\]
Therefore,
\[
\sigma_f(t)\le
C_1kd\,p^{\frac{d(2k-1)}{2}\varkappa}
\bigl(a^{-1}b^{\frac{2k-1}{2}}t\bigr)^{\alpha_p},
\]
where $C_1\ge1$ is an absolute constant.

We note that
\[
\varkappa - \alpha_p
=
\frac{(2k-1)p^{-1}}{2k(d-1)\bigl(2k(d-1)+(2k-1)p^{-1}\bigr)}
\le p^{-1}.
\]
Set
\[
A := a^{-1}b^{\frac{2k-1}{2}}t
\quad\text{and}\quad
s := \frac{d(2k-1)}{2}\varkappa=\frac{d(2k-1)}{4k(d-1)}.
\]
Fix $\alpha \in (0,\varkappa)$.

If $A \le 1$, choose
\[
p = \frac{1}{\varkappa - \alpha}
\ge \frac{1}{\varkappa}
= 2k(d-1)
\ge 2k.
\]
Then
\[
\varkappa - \alpha_p \le p^{-1}= \varkappa - \alpha,
\]
that is,
$\alpha \le \alpha_p$.
Therefore,
\[
\sigma_f(t)
\le
C_1kd\,p^sA^{\alpha_p}
\le
C_1kd(\varkappa-\alpha)^{-s}A^\alpha.
\]

If $A \ge 1$, then for every $\alpha \in (0,\varkappa)$,
\[
\sigma_f(t)\le 1 \le C_1kd(\varkappa-\alpha)^{-s}A^\alpha.
\]
Hence, for all $t>0$ and all $\alpha \in (0,\varkappa)$,
\[
\sigma_f(t)
\le
C_1kd(\varkappa-\alpha)^{-s}
\bigl(a^{-1}b^{\frac{2k-1}{2}}t\bigr)^\alpha.
\]

Next, assume that
$
A < e^{-1/\varkappa}.
$
Then
\[
\alpha = \varkappa - \frac{1}{\ln A^{-1}} \in (0,\varkappa),
\]
and therefore
\[
\sigma_f(t)
\le
C_1kd(\ln A^{-1})^sA^{\varkappa-\frac{1}{\ln A^{-1}}}
=
eC_1kd(\ln A^{-1})^sA^\varkappa.
\]
If
$A \ge e^{-1/\varkappa}$,
then
\[
\sigma_f(t)\le 1 \le eC_1kd\,A^\varkappa
\le
eC_1kd\bigl(1+|\ln A^{-1}|\bigr)^sA^\varkappa.
\]
Thus, for all $t>0$,
\[
\sigma_f(t)
\le
eC_1kd
\bigl(1+\bigl|\ln\bigl(a^{-1}b^{\frac{2k-1}{2}}t\bigr)\bigr|\bigr)^s
\bigl(a^{-1}b^{\frac{2k-1}{2}}t\bigr)^\varkappa,
\]
which completes the proof.
\end{proof}

\section{Bounds in terms of Kantorovich-type distances}
\label{sec:KR}

\begin{theorem}\label{t2.1}
Let $k\in\mathbb{N}$, let $p\ge 2k$, let
\[
q:=\frac{p}{p-2k+1},
\]
let $\varkappa\in(0,q)$, and let $a,b>0$. For $r\in\mathbb{N}$, set
\[
\beta_r:=
\frac{p\varkappa}{p(r+\varkappa)+\varkappa r(2k-1)}.
\]
Then, for every pair of mappings $f=(f_1,\ldots,f_k)$ and $g=(g_1,\ldots,g_k)$ with
$f_j,g_j\in W^{2,p}(\gamma)$,
\[
\max_{1\le j\le k}\|f_j\|_{\dot W^{2,p}(\gamma)}\le b,
\quad
\max_{1\le j\le k}\|g_j\|_{\dot W^{2,p}(\gamma)}\le b,
\]
and
\[
\gamma(\Delta_f\le \varepsilon)\le a\varepsilon^\varkappa,
\quad
\gamma(\Delta_g\le \varepsilon)\le a\varepsilon^\varkappa
\quad \forall\, \varepsilon>0,
\]
one has
\[
d_{\mathrm{TV}}(f,g)
\le
2^{14}k^{7/2}(q-\varkappa)^{-1}(a+b^p+1)\,
d_{\mathrm{KR}}(f,g)^{\beta_1}.
\]
More generally, for every $r\in\mathbb{N}$,
\[
d_{\mathrm{TV}}(f,g)
\le
C(r,k)(q-\varkappa)^{-1}(a+b^p+1)\,
d_r(f,g)^{\beta_r},
\]
where
$C(r,k)>0$ depends only on $r$ and $k$.
\end{theorem}

\begin{proof}
Let
\[
\alpha:=
\frac{p\varkappa}{p+\varkappa(2k-1)}.
\]
By Theorem \ref{t1.1}, applied to $f$ and $g$, we have
\[
\sigma_f(t)\le Mt^\alpha
\quad \text{and} \quad
\sigma_g(t)\le Mt^\alpha
\quad \forall\, t>0,
\]
where
\[
M:=512\,k^\alpha(1-\alpha)^{-\alpha}
a^{\alpha/\varkappa}b^{(2k-1)\alpha}.
\]
Hence, by \eqref{eq:intro-TV-est}, for every $t\in(0,1]$,
\[
d_{\mathrm{TV}}(f,g)
\le
6\sqrt{k}\,Mt^\alpha+\sqrt{k}\,t^{-1}d_{\mathrm{KR}}(f,g).
\]
Since $d_{\mathrm{KR}}(f,g)\le 2$, we may take
\[
t=\bigl(\tfrac12 d_{\mathrm{KR}}(f,g)\bigr)^{\frac{1}{1+\alpha}}.
\]
Then
\[
d_{\mathrm{TV}}(f,g)
\le
(6\sqrt{k}\,M+2\sqrt{k})
d_{\mathrm{KR}}(f,g)^{\frac{\alpha}{1+\alpha}}.
\]

Since $\alpha\le 1$,
\begin{align}\label{eq-alpha-est}
(1-\alpha)^{-\alpha}
&\le
(1-\alpha)^{-1}
=
\frac{p+\varkappa(2k-1)}{p-\varkappa(p-2k+1)}
\le
\frac{p+\frac{p(2k-1)}{p-2k+1}}{p-\varkappa(p-2k+1)}\\
&=
\bigl(\tfrac{p}{p-2k+1}\bigr)^2
\bigl(\tfrac{p}{p-2k+1}-\varkappa\bigr)^{-1}
\le
4k^2
\bigl(\tfrac{p}{p-2k+1}-\varkappa\bigr)^{-1}
=4k^2
(q-\varkappa)^{-1}.
\nonumber
\end{align}
Moreover, by Young's inequality,
\begin{equation}\label{eq-Young}
	a^{\alpha/\varkappa}b^{(2k-1)\alpha}\le a+b^p.
\end{equation}
Therefore,
\[
6\sqrt{k}\,M
\le
2^{14}k^{7/2}
(q-\varkappa)^{-1}
(a+b^p).
\]
Since
\[
4k^2(q-\varkappa)^{-1}\ge 1,
\]
we also have
\[
2\sqrt{k}
\le
2^{14}k^{7/2}
(q-\varkappa)^{-1}.
\]
Hence
\[
6\sqrt{k}\,M+2\sqrt{k}
\le
2^{14}k^{7/2}
(q-\varkappa)^{-1}
(a+b^p+1),
\]
and since $\frac{\alpha}{1+\alpha}=\beta$, we obtain
\[
d_{\mathrm{TV}}(f,g)
\le
2^{14}k^{7/2}
(q-\varkappa)^{-1}
(a+b^p+1)\,
d_{\mathrm{KR}}(f,g)^\beta.
\]

For the estimate in terms of $d_r$, by \eqref{eq-TV-bound}, for every $t\in(0,1]$,
\[
d_{\mathrm{TV}}(f,g)
\le
4\sqrt{k}\,Mt^\alpha+C_r(k)t^{-r}d_r(f,g).
\]
Since $d_r(f,g)\le 2$, we may take
\[
t=\bigl(\tfrac12 d_r(f,g)\bigr)^{\frac{1}{r+\alpha}}.
\]
Then
\[
d_{\mathrm{TV}}(f,g)
\le
\bigl(4\sqrt{k}\,M+2C_r(k)\bigr)d_r(f,g)^{\beta_r}.
\]
Using \eqref{eq-alpha-est} and \eqref{eq-Young}, we obtain
\[
d_{\mathrm{TV}}(f,g)
\le
C(r,k)(q-\varkappa)^{-1}(a+b^p+1)\,
d_r(f,g)^{\beta_r},
\]
which completes the proof.
\end{proof}

\begin{theorem}\label{t2.2}
Let $k,d\in\mathbb{N}$ with $d\ge 2$, and let $a,b>0$. Set
\[
\beta_1:=\frac{1}{2k(d-1)+1}
\quad\text{and}\quad
s:=\frac{d(2k-1)}{4k(d-1)}.
\]
Then, for every pair of mappings $f=(f_1,\ldots,f_k)$ and $g=(g_1,\ldots,g_k)$ with
$f_j,g_j\in \mathcal{P}_d(\gamma)$,
\[
\int_E \Delta_f\,d\gamma\ge a,
\quad
\int_E \Delta_g\,d\gamma\ge a,
\]
and
\[
\max_{1\le j\le k}\operatorname{Var}_\gamma(f_j)\le b,
\quad
\max_{1\le j\le k}\operatorname{Var}_\gamma(g_j)\le b,
\]
one has, for every $\beta\in(0,\beta_1)$,
\begin{equation}\label{est-1}
d_{\mathrm{TV}}(f,g)
\le
C_1k^{3/2}d\,(\beta_1-\beta)^{-s}
\bigl(1+a^{-\frac{1}{2k(d-1)}}b^{\frac{2k-1}{4k(d-1)}}\bigr)\,
d_{\mathrm{KR}}(f,g)^\beta,
\end{equation}
where $C_1>0$ is an absolute constant. In particular, if $d_{\mathrm{KR}}(f,g)>0$, then
\begin{equation}\label{est-2}
d_{\mathrm{TV}}(f,g)
\le
C_2k^{3/2}d
\bigl(1+a^{-\frac{1}{2k(d-1)}}b^{\frac{2k-1}{4k(d-1)}}\bigr)
\bigl(1+\ln\tfrac{2}{d_{\mathrm{KR}}(f,g)}\bigr)^s
d_{\mathrm{KR}}(f,g)^{\beta_1},
\end{equation}
where $C_2>0$ is an absolute constant.

More generally, for every $r\in\mathbb{N}$, setting
\[
\beta_r:=\frac{1}{2rk(d-1)+1},
\]
one has, for every $\beta\in(0,\beta_r)$,
\begin{equation}\label{est-3}
d_{\mathrm{TV}}(f,g)
\le
C_1(r,k)\,d\,(\beta_r-\beta)^{-s}
\bigl(1+a^{-\frac{1}{2k(d-1)}}b^{\frac{2k-1}{4k(d-1)}}\bigr)\,
d_r(f,g)^\beta,
\end{equation}
where $C_1(r,k)>0$ depends only on $r$ and $k$. Moreover, if $d_r(f,g)>0$, then
\begin{equation}\label{est-4}
d_{\mathrm{TV}}(f,g)
\le
C_2(r,k)\,d
\bigl(1+a^{-\frac{1}{2k(d-1)}}b^{\frac{2k-1}{4k(d-1)}}\bigr)
\bigl(1+\ln\tfrac{2}{d_r(f,g)}\bigr)^s
d_r(f,g)^{\beta_r},
\end{equation}
where $C_2(r,k)>0$ depends only on $r$ and $k$.
\end{theorem}

\begin{proof}
Set
\[
\varkappa:=\frac{1}{2k(d-1)}
\quad\text{and}\quad
B:=a^{-1}b^{\frac{2k-1}{2}}.
\]
Then
\[
\beta_1=\frac{\varkappa}{1+\varkappa}.
\]
By Theorem \ref{t1.2}, for every $\alpha\in(0,\varkappa)$ and every $t>0$,
\[
\sigma_f(t)\le M_\alpha t^\alpha
\quad\text{and}\quad
\sigma_g(t)\le M_\alpha t^\alpha,
\]
where
\[
M_\alpha:=Ckd(\varkappa-\alpha)^{-s}B^\alpha
\]
and $C>0$ is an absolute constant. Applying \eqref{eq:intro-TV-est} with
\[
t=\bigl(\tfrac12 d_{\mathrm{KR}}(f,g)\bigr)^{\frac{1}{1+\alpha}},
\]
we obtain
\[
d_{\mathrm{TV}}(f,g)
\le
\bigl(6\sqrt{k}\,M_\alpha+2\sqrt{k}\bigr)
d_{\mathrm{KR}}(f,g)^{\frac{\alpha}{1+\alpha}}.
\]
Now let $\beta\in(0,\beta_1)$ and choose $\alpha\in(0,\varkappa)$ so that
\[
\beta=\frac{\alpha}{1+\alpha}.
\]
Since
\[
\beta_1-\beta=\frac{\varkappa-\alpha}{(1+\varkappa)(1+\alpha)}\le \varkappa-\alpha,
\]
we have
\[
(\varkappa-\alpha)^{-s}\le (\beta_1-\beta)^{-s}.
\]
Moreover,
\begin{equation}\label{eq-est-B}
B^\alpha\le 1+B^\varkappa.
\end{equation}
Therefore,
\[
6\sqrt{k}\,M_\alpha+2\sqrt{k}
\le
(6C+2)\,k^{3/2}d\,(\beta_1-\beta)^{-s}(1+B^\varkappa),
\]
and hence
\begin{equation}\label{eq-first-est}
d_{\mathrm{TV}}(f,g)
\le
C_1k^{3/2}d\,(\beta_1-\beta)^{-s}(1+B^\varkappa)
d_{\mathrm{KR}}(f,g)^\beta,
\end{equation}
where $C_1:=6C+2\ge2$.
This proves \eqref{est-1}.

To prove \eqref{est-2}, assume first that
\[
d_{\mathrm{KR}}(f,g)<2e^{-1/\beta_1}.
\]
Then we may take
\[
\beta=\beta_1-\frac{1}{\ln\bigl(2/d_{\mathrm{KR}}(f,g)\bigr)}.
\]
By \eqref{eq-first-est},
\[
d_{\mathrm{TV}}(f,g)
\le
C_1k^{3/2}d
\bigl(\ln\tfrac{2}{d_{\mathrm{KR}}(f,g)}\bigr)^s
(1+B^\varkappa)
d_{\mathrm{KR}}(f,g)^{\beta_1-\frac{1}{\ln(2/d_{\mathrm{KR}}(f,g))}}.
\]
Since
\[
d_{\mathrm{KR}}(f,g)^{-\frac{1}{\ln(2/d_{\mathrm{KR}}(f,g))}}\le e,
\]
it follows that
\[
d_{\mathrm{TV}}(f,g)
\le
eC_1k^{3/2}d
\bigl(\ln\tfrac{2}{d_{\mathrm{KR}}(f,g)}\bigr)^s
(1+B^\varkappa)
d_{\mathrm{KR}}(f,g)^{\beta_1}.
\]
	
If
\[
d_{\mathrm{KR}}(f,g)\ge 2e^{-1/\beta_1},
\]
then
\[
d_{\mathrm{KR}}(f,g)^{\beta_1}\ge 2^{\beta_1}e^{-1}\ge e^{-1},
\]
and therefore
\[
d_{\mathrm{TV}}(f,g)
\le 2 \le 2e\,d_{\mathrm{KR}}(f,g)^{\beta_1}
\le
eC_1k^{3/2}d
\bigl(1+\ln\tfrac{2}{d_{\mathrm{KR}}(f,g)}\bigr)^s
(1+B^\varkappa)
d_{\mathrm{KR}}(f,g)^{\beta_1}.
\]
This proves the estimate \eqref{est-2} with $C_2=eC_1$.
	
For the estimate \eqref{est-3} and \eqref{est-4}, we argue in the same way, applying \eqref{eq-TV-bound} with
\[
t=\bigl(\tfrac12 d_r(f,g)\bigr)^{\frac{1}{r+\alpha}},
\]
which yields
\[
d_{\mathrm{TV}}(f,g)
\le
\bigl(4\sqrt{k}\,M_\alpha+2C_r(k)\bigr)d_r(f,g)^{\frac{\alpha}{r+\alpha}}.
\]
Now let $\beta\in(0,\beta_r)$ and choose $\alpha\in(0,\varkappa)$ so that
\[
\beta=\frac{\alpha}{r+\alpha}.
\]
Since
\[
\beta_r-\beta=\frac{r(\varkappa-\alpha)}{(r+\varkappa)(r+\alpha)}\le \varkappa-\alpha,
\]
we have
\[
(\varkappa-\alpha)^{-s}\le (\beta_r-\beta)^{-s}.
\]
Using \eqref{eq-est-B}, we obtain
\[
d_{\mathrm{TV}}(f,g)
\le
C_3(r,k)\,d\,(\beta_r-\beta)^{-s}(1+B^\varkappa)d_r(f,g)^\beta
\]
for some constant $C_3(r,k)>0$ depending only on $r$ and $k$.
	
Finally, if
\[
d_r(f,g)<2e^{-1/\beta_r},
\]
we take
\[
\beta=\beta_r-\frac{1}{\ln\bigl(2/d_r(f,g)\bigr)}
\]
and argue as above. If
\[
d_r(f,g)\ge 2e^{-1/\beta_r},
\]
then
\[
d_r(f,g)^{\beta_r}\ge e^{-1}.
\]
Thus, in both cases,
\[
d_{\mathrm{TV}}(f,g)
\le
C_4(r,k)\,d
\bigl(1+\ln\tfrac{2}{d_r(f,g)}\bigr)^s
(1+B^\varkappa)d_r(f,g)^{\beta_r}
\]
for some constant $C_4(r,k)>0$ depending only on $r$ and $k$.
This completes the proof.
\end{proof}

\section{Bounds in terms of Sobolev and $L^2$ distances}
\label{sec:L2}

\begin{lemma}\label{lem6.0}
Let $p\ge 2k$, and let
\[
f=(f_1,\ldots,f_k)
\quad\text{and}\quad
g=(g_1,\ldots,g_k)
\]
be such that
$
f_j,g_j\in C_b^\infty(\mathbb R^n),
$
and
\[
\max_{1\le j\le k}\|\nabla f_j\|_{L^p(\gamma_n)}\le b
\quad\text{and}\quad
\max_{1\le j\le k}\|\nabla g_j\|_{L^p(\gamma_n)}\le b.
\]
Then
\[
\Delta_f\le \prod_{j=1}^k|\nabla f_j|^2,
\]
and
\[
\|\Delta_f-\Delta_g\|_{L^{\frac{p}{2k}}(\gamma_n)}
\le
4k\,b^{2k-1}
\max_{1\le j\le k}\|\nabla f_j-\nabla g_j\|_{L^p(\gamma_n)}.
\]
\end{lemma}

\begin{proof}
For a mapping $f=(f_1,\ldots,f_k)$, set
\[
J_f^*:=(\nabla f_1,\ldots,\nabla f_k),
\]
so that
\[
\Delta_f=\det(J_fJ_f^*).
\]
We note that $\sqrt{\det(J_fJ_f^*)}$ is the volume of the parallelepiped spanned by the columns of $J_f^*$. Therefore,
\begin{equation}\label{eq-Hadamard}
	\sqrt{\det(J_fJ_f^*)}\le \prod_{j=1}^k |\nabla f_j|,
\end{equation}
which is equivalent to the first claimed estimate.

To prove the second estimate,
for $i=0,\ldots,k$, we set
\[
w^i:=(g_1,\ldots,g_i,f_{i+1},\ldots,f_k),
\]
so that $w^0=f$ and $w^k=g$, and
\[
v^i:=(g_1,\ldots,g_{i-1},f_i-g_i,f_{i+1},\ldots,f_k).
\]
Then
\begin{align*}
|\Delta_f-\Delta_g|
&=
\bigl|\det(J_{w^0}J_{w^0}^*)-\det(J_{w^k}J_{w^k}^*)\bigr|
\\
&\le
\bigl|\det(J_{w^0}J_{w^0}^*)-\det(J_{w^0}J_{w^k}^*)\bigr|
+
\bigl|\det(J_{w^0}J_{w^k}^*)-\det(J_{w^k}J_{w^k}^*)\bigr|
\\
&\le
\sum_{i=1}^k
\bigl|\det(J_{w^0}J_{w^{i-1}}^*)-\det(J_{w^0}J_{w^i}^*)\bigr|
+
\sum_{i=1}^k
\bigl|\det(J_{w^{i-1}}J_{w^k}^*)-\det(J_{w^i}J_{w^k}^*)\bigr|
\\
&=
\sum_{i=1}^k
\bigl|\det(J_{w^0}J_{v^i}^*)\bigr|
+
\sum_{i=1}^k
\bigl|\det(J_{v^i}J_{w^k}^*)\bigr|.
\end{align*}
From the Cauchy--Binet formula and the Cauchy--Schwarz inequality, one deduces that
\[
\bigl|\det(J_{w^0}J_{v^i}^*)\bigr|
\le
\sqrt{\det(J_{w^0}J_{w^0}^*)}
\sqrt{\det(J_{v^i}J_{v^i}^*)}.
\]
Moreover, by \eqref{eq-Hadamard},
\[
\sqrt{\det(J_{w^0}J_{w^0}^*)}\le \prod_{j=1}^k |\nabla w_j^0|,
\qquad
\sqrt{\det(J_{v^i}J_{v^i}^*)}\le \prod_{j=1}^k |\nabla v_j^i|.
\]
Therefore,
\[
\bigl|\det(J_{w^0}J_{v^i}^*)\bigr|
\le
\prod_{j=1}^k |\nabla v_j^i|\,|\nabla w_j^0|.
\]
Similarly,
\[
\bigl|\det(J_{v^i}J_{w^k}^*)\bigr|
\le
\prod_{j=1}^k |\nabla v_j^i|\,|\nabla w_j^k|.
\]
Since
\[
|\nabla w_j^i|\le \max\{|\nabla f_j|,|\nabla g_j|\}\quad \forall j\in\{1, \ldots, k\},
\qquad
|\nabla v_j^i|\le \max\{|\nabla f_j|,|\nabla g_j|\}
\quad \forall j\ne i,
\]
and
\[
|\nabla v_i^i|=|\nabla f_i-\nabla g_i|,
\]
we obtain
\[
|\Delta_f-\Delta_g|
\le
2
\sum_{i=1}^k |\nabla f_i-\nabla g_i|\max\{|\nabla f_i|,|\nabla g_i|\}\prod_{j\ne i}\max\{|\nabla f_j|^2,|\nabla g_j|^2\}.
\]
We now estimate the norm using H\"older's inequality:
\begin{align*}
&\|\Delta_f-\Delta_g\|_{L^{\frac{p}{2k}}(\gamma_n)}
\le
2\sum_{i=1}^k \Bigl\||\nabla f_i-\nabla g_i|\max\{|\nabla f_i|,|\nabla g_i|\}\prod_{j\ne i}\max\{|\nabla f_j|^2,|\nabla g_j|^2\}
\Bigr\|_{L^{\frac{p}{2k}}(\gamma_n)}
\\
&\le
2\sum_{i=1}^k \bigl\|\nabla f_i-\nabla g_i\bigr\|_{L^p(\gamma_n)}\bigl\|\max\{|\nabla f_i|,|\nabla g_i|\}\bigr\|_{L^p(\gamma_n)}\prod_{j\ne i}\bigl\|\max\{|\nabla f_j|,|\nabla g_j|\}
\bigr\|_{L^p(\gamma_n)}^2
\end{align*}
Since
\[
\bigl\|\max\{|\nabla f_j|,|\nabla g_j|\}
\bigr\|_{L^p(\gamma_n)}^p\le 
\bigl\|\nabla f_j\bigr\|_{L^p(\gamma_n)}^p
+
\bigl\|\nabla g_j\bigr\|_{L^p(\gamma_n)}^p
\le 2b^p,
\]
we obtain
\[
\|\Delta_f-\Delta_g\|_{L^{\frac{p}{2k}}(\gamma_n)}\le 
2^{1+\frac{2k-1}{p}}kb^{2k-1}\max_{1\le i\le k}\bigl\|\nabla f_i-\nabla g_i\bigr\|_{L^p(\gamma_n)},
\]
which yields the claimed estimate.
\end{proof}

\begin{lemma}\label{lem1.11}
Let $k\in\mathbb N$, let $p>4k$, and let $b>0$.
Let
\[
f=(f_1,\ldots,f_k),\quad g=(g_1,\ldots,g_k)\colon \mathbb R^n\to\mathbb R^k
\]
be such that
\[
f_j,g_j\in C_b(\mathbb{R}^n),
\quad
\max_{1\le j\le k}\|\nabla f_j\|_{L^p(\gamma_n)}\le b,
\quad
\max_{1\le j\le k}\|\nabla g_j\|_{L^p(\gamma_n)}\le b.
\]
Then, for every pair of functions
\[
u\in L^{\frac{p}{2k-1}}(\gamma_n)
\quad\text{and}\quad
v\in L^p(\gamma_n),
\]
one has	
\begin{align*}
&\int_{\mathbb R^n} (\Delta_g+\varepsilon)^{-1}|uv|\,d\gamma_n
\le
\|u\|_{L^{\frac{p}{2k-1}}(\gamma_n)}
\|v\|_{L^p(\gamma_n)}
\Bigl(\int_{\mathbb R^n} (\Delta_f+\varepsilon)^{-m_1}\,d\gamma_n\Bigr)^{1/m_1}
\\
&+
4k\varepsilon^{-1}b^{2k-1}
\|u\|_{L^{\frac{p}{2k-1}}(\gamma_n)}
\|v\|_{L^p(\gamma_n)}
\max_{1\le j\le k}\|\nabla f_j-\nabla g_j\|_{L^p(\gamma_n)}
\Bigl(\int_{\mathbb R^n} (\Delta_f+\varepsilon)^{-m_2}\,d\gamma_n\Bigr)^{1/m_2},
\end{align*}
where
\[
m_1:=\frac{p}{p-2k}
\quad\text{and}\quad
m_2:=\frac{p}{p-4k}.
\]
\end{lemma}

\begin{proof}
We have
\[
\frac{1}{\Delta_g+\varepsilon}
=
\frac{1}{\Delta_g+\varepsilon}
-
\frac{1}{\Delta_f+\varepsilon}
+
\frac{1}{\Delta_f+\varepsilon}
\le
\frac{|\Delta_f-\Delta_g|}{\varepsilon(\Delta_f+\varepsilon)}
+
\frac{1}{\Delta_f+\varepsilon}.
\]
Therefore,
\begin{equation}\label{eq-lem-11}
\int_{\mathbb R^n} (\Delta_g+\varepsilon)^{-1}|uv|\,d\gamma_n
\le
\int_{\mathbb R^n} \frac{|\Delta_f-\Delta_g||uv|}{\varepsilon(\Delta_f+\varepsilon)}\,d\gamma_n
+
\int_{\mathbb R^n} \frac{|uv|}{\Delta_f+\varepsilon}\,d\gamma_n.
\end{equation}
Applying Lemma \ref{lem6.0} and then H\"older's inequality, we obtain
\begin{align*}
&\int_{\mathbb R^n} \frac{|\Delta_f-\Delta_g||uv|}{\varepsilon(\Delta_f+\varepsilon)}\,d\gamma_n
\le
\varepsilon^{-1}
\|\Delta_f-\Delta_g\|_{L^{\frac{p}{2k}}(\gamma_n)}
\|u\|_{L^{\frac{p}{2k-1}}(\gamma_n)}
\|v\|_{L^p(\gamma_n)}
\Bigl(\int_{\mathbb R^n} (\Delta_f+\varepsilon)^{-m_2}\,d\gamma_n\Bigr)^{1/m_2}
\\
&\le
4k\varepsilon^{-1}b^{2k-1}
\|u\|_{L^{\frac{p}{2k-1}}(\gamma_n)}
\|v\|_{L^p(\gamma_n)}
\max_{1\le j\le k}\|\nabla f_j-\nabla g_j\|_{L^p(\gamma_n)}
\Bigl(\int_{\mathbb R^n} (\Delta_f+\varepsilon)^{-m_2}\,d\gamma_n\Bigr)^{1/m_2}.
\end{align*}
For the second term in \eqref{eq-lem-11}, H\"older's inequality yields
\[
\int_{\mathbb R^n} \frac{|uv|}{\Delta_f+\varepsilon}\,d\gamma_n
\le
\|u\|_{L^{\frac{p}{2k-1}}(\gamma_n)}
\|v\|_{L^p(\gamma_n)}
\Bigl(\int_{\mathbb R^n} (\Delta_f+\varepsilon)^{-m_1}\,d\gamma_n\Bigr)^{1/m_1}.
\]
Combining the two estimates completes the proof.
\end{proof}

\begin{theorem}\label{t3.1}
Let $k\in\mathbb N$, let $p>6k$, let $\varkappa\in(0,1)$, and let $a,b>0$. Set
\[
\beta:=\frac{p\varkappa}{p+2k\varkappa}.
\]
Assume that $f=(f_1,\ldots,f_k)$ and $g=(g_1,\ldots,g_k)$
with $
f_j,g_j\in W^{2,p}(\gamma)$,
satisfy
\[
\max_{1\le j\le k}\|f_j\|_{\dot W^{2,p}(\gamma)}\le b,
\quad
\max_{1\le j\le k}\|g_j\|_{\dot W^{2,p}(\gamma)}\le b,
\]
and
\[
\gamma(\Delta_f\le s)\le as^\varkappa
\quad \forall\, s>0.
\]
Then
\[
d_{\rm TV}(f,g)
\le
176k^4(1-\varkappa)^{-1}
\bigl(a^{1/\varkappa}b^{2k-1}\max_{1\le j\le k}\|f_j-g_j\|_{W^{1,p}(\gamma)}\bigr)^\beta.
\]
\end{theorem}

\begin{proof}
Fix $\varphi\in C_0^\infty(\mathbb R^k)$ with $\|\varphi\|_\infty\le 1$, and let $\varepsilon>0$ be chosen later.

\medskip
\noindent
\textbf{Step 1: the cylindrical case.}
Similarly to the proof of Lemma \ref{lem:main},
we first assume that $g_1, \ldots, g_k,f_1,\ldots,f_k\in \mathcal{FC}^\infty$ and that
\[
g_j(x)=\widetilde g_j\bigl(\ell_1(x),\ldots,\ell_n(x)\bigr),
\quad
f_j(x)=\widetilde f_j\bigl(\ell_1(x),\ldots,\ell_n(x)\bigr),
\quad j\in\{1,\ldots,k\},
\]
for some $\ell_1,\ldots,\ell_n\in E^*$.
Without loss of generality, we may assume that
$\{\ell_1,\ldots,\ell_n\}$ is an orthonormal system in $L^2(\gamma)$, that
$\{h_{\ell_1},\ldots,h_{\ell_n}\}$ is an orthonormal system in $H(\gamma)$, and
that the distribution of the mapping
\[
(\ell_1,\ldots,\ell_n)\colon (E,\gamma)\to\mathbb R^n
\]
is the standard Gaussian measure $\gamma_n$ on $\mathbb R^n$.

For $i=0,\ldots,k$, we set
\[
w^i:=(\widetilde{g}_1,\ldots,\widetilde{g}_i,\widetilde{f}_{i+1},\ldots,\widetilde{f}_k),
\]
so that $w^0=\widetilde{f}$ and $w^k=\widetilde{g}$.
We also set
\[
T:=\max_{1\le j\le k}\max\bigl\{\|f_j\|_{\dot W^{2,p}(\gamma)}^{2k-1}, \|g_j\|_{\dot W^{2,p}(\gamma)}^{2k-1}\bigr\}\max_{1\le j\le k}\|f_j-g_j\|_{W^{1,p}(\gamma)}.
\]
We note that, for every pair of indexes $0\le i_1<i_2\le k$,
\[
\max_{1\le j\le k}\max\bigl\{\|w^{i_1}_j -  w^{i_2}_j\|_{L^p(\gamma_n)}, \|\nabla w^{i_1}_j - \nabla w^{i_2}_j\|_{L^p(\gamma_n)}\bigr\}\le 
\max_{1\le j\le k}\|f_j-g_j\|_{W^{1,p}(\gamma)}.
\]

After changing the measure, we have
\begin{align}\label{eq-tv-decomp}
\int_E (\varphi(f)-\varphi(g))\,d\gamma
&=
\int_{\mathbb{R}^n} (\varphi(\widetilde{f})-\varphi(\widetilde{g}))\,d\gamma_n
\\
&=
\sum_{i=1}^k \int_{\mathbb{R}^n} \varphi(w^i)
\Bigl(
\frac{\Delta_{w^{i-1}}}{\Delta_{w^{i-1}}+\varepsilon}
-
\frac{\Delta_{w^i}}{\Delta_{w^i}+\varepsilon}
\Bigr)\,d\gamma_n
\nonumber
\\
&+
\sum_{i=1}^k \int_{\mathbb{R}^n}
(\varphi(w^{i-1})-\varphi(w^i))
\frac{\Delta_{w^{i-1}}}{\Delta_{w^{i-1}}+\varepsilon}\,d\gamma_n
\nonumber
\\
&+
\int_{\mathbb{R}^n} \varphi(w^0)\frac{\varepsilon}{\Delta_{w^0}+\varepsilon}\,d\gamma_n
-
\int_{\mathbb{R}^n} \varphi(w^k)\frac{\varepsilon}{\Delta_{w^k}+\varepsilon}\,d\gamma_n.
\nonumber
\end{align}

\medskip
\noindent
\textbf{Step 2: estimate of the first sum in \eqref{eq-tv-decomp}.}
To estimate the first sum in \eqref{eq-tv-decomp}
we note that	
\begin{align*}
&\Bigl|\frac{\Delta_{w^{i-1}}}{\Delta_{w^{i-1}}+\varepsilon}  -
\frac{\Delta_{w^{i}}}{\Delta_{w^{i}}+\varepsilon} \Bigr| =
\frac{\varepsilon|\Delta_{w^{i-1}} - \Delta_{w^{i}}|}{(\Delta_{w^{i-1}}+\varepsilon)(\Delta_{w^{i}}+\varepsilon)} 
\\ 
&\le
\varepsilon|\Delta_{w^{i-1}} - \Delta_{w^{i}}|
\Bigl( \frac{|\Delta_{w^0} - \Delta_{w^{i-1}}|}{(\Delta_{w^{i-1}}+\varepsilon)(\Delta_{w^0}+\varepsilon)}  +\frac{1}{\Delta_{w^0}+\varepsilon}\Bigr) \Bigl( \frac{|\Delta_{w^0} - \Delta_{w^{i}}|}{(\Delta_{w^{i}}+\varepsilon)(\Delta_{w^0}+\varepsilon)}  +\frac{1}{\Delta_{w^0}+\varepsilon} \Bigr),
\end{align*}
Hence
\begin{align*}
\biggl|
\int_{\mathbb{R}^n} \varphi(w^i)
&\Bigl(
\frac{\Delta_{w^{i-1}}}{\Delta_{w^{i-1}}+\varepsilon}
-
\frac{\Delta_{w^i}}{\Delta_{w^i}+\varepsilon}
\Bigr)\,d\gamma_n
\biggr|
\\
&\le
\varepsilon^{-1}
\int_{\mathbb{R}^n}
\frac{|\Delta_{w^{i-1}}-\Delta_{w^i}||\Delta_{w^0}-\Delta_{w^{i-1}}||\Delta_{w^0}-\Delta_{w^i}|}
{(\Delta_{w^0}+\varepsilon)^2}\,d\gamma_n
\\
&+
\int_{\mathbb{R}^n}
\frac{|\Delta_{w^{i-1}}-\Delta_{w^i}||\Delta_{w^0}-\Delta_{w^{i-1}}|}{(\Delta_{w^0}+\varepsilon)^2}\,d\gamma_n
+
\int_{\mathbb{R}^n}
\frac{|\Delta_{w^{i-1}}-\Delta_{w^i}||\Delta_{w^0}-\Delta_{w^i}|}{(\Delta_{w^0}+\varepsilon)^2}\,d\gamma_n
\\
&+
\varepsilon
\int_{\mathbb{R}^n}
\frac{|\Delta_{w^{i-1}}-\Delta_{w^i}|}{(\Delta_{w^0}+\varepsilon)^2}\,d\gamma_n.
\end{align*}
Set
\[
m_1:=\frac{p}{p-2k},
\qquad
m_2:=\frac{p}{p-4k},
\qquad
m_3:=\frac{p}{p-6k}.
\]
By Lemma \ref{lem6.0} and H\"older's inequality we obtain
\begin{align*}
&\int_{\mathbb{R}^n}
\frac{|\Delta_{w^{i-1}}-\Delta_{w^i}||\Delta_{w^0}-\Delta_{w^{i-1}}||\Delta_{w^0}-\Delta_{w^i}|}
{(\Delta_{w^0}+\varepsilon)^2}\,d\gamma_n
\\
&\le
\varepsilon^{-1}\|\Delta_{w^{i-1}}-\Delta_{w^i}\|_{L^{\frac{p}{2k}}(\gamma_n)}
\|\Delta_{w^0}-\Delta_{w^{i-1}}\|_{L^{\frac{p}{2k}}(\gamma_n)}
\|\Delta_{w^0}-\Delta_{w^i}\|_{L^{\frac{p}{2k}}(\gamma_n)}
\Bigl(\int_{\mathbb{R}^n}(\Delta_{w^0}+\varepsilon)^{-m_3}\,d\gamma_n\Bigr)^{1/m_3}
\\
&\le
64 k^3T^3\varepsilon^{-1} 
\Bigl(\int_{\mathbb{R}^n}(\Delta_{w^0}+\varepsilon)^{-m_3}\,d\gamma_n\Bigr)^{1/m_3}.
\end{align*}
Similarly, we have
\begin{align*}
&\int_{\mathbb{R}^n}
\frac{|\Delta_{w^{i-1}}-\Delta_{w^i}||\Delta_{w^0}-\Delta_{w^{i-1}}|}{(\Delta_{w^0}+\varepsilon)^2}\,d\gamma_n
+
\int_{\mathbb{R}^n}
\frac{|\Delta_{w^{i-1}}-\Delta_{w^i}||\Delta_{w^0}-\Delta_{w^i}|}{(\Delta_{w^0}+\varepsilon)^2}\,d\gamma_n
\\
&\le 	
32 k^2T^2 \varepsilon^{-1}
\Bigl(\int_{\mathbb{R}^n}(\Delta_{w^0}+\varepsilon)^{-m_2}\,d\gamma_n\Bigr)^{1/m_2}
\end{align*}	
and
\[
\int_{\mathbb{R}^n}
\frac{|\Delta_{w^{i-1}}-\Delta_{w^i}|}{(\Delta_{w^0}+\varepsilon)^2}\,d\gamma_n
\le 
4kT\varepsilon^{-1}
\Bigl(\int_{\mathbb{R}^n}(\Delta_{w^0}+\varepsilon)^{-m_1}\,d\gamma_n\Bigr)^{1/m_1}
\]
Summing up, we obtain
\begin{align*}
\sum_{i=1}^k
\biggl|
\int_{\mathbb R^n} \varphi(w^i)
\Bigl(
\frac{\Delta_{w^{i-1}}}{\Delta_{w^{i-1}}+\varepsilon}
-
\frac{\Delta_{w^i}}{\Delta_{w^i}+\varepsilon}
\Bigr)\,d\gamma_n
\biggr|
&\le
64k^4T^3\varepsilon^{-2}
\Bigl(\int_{\mathbb R^n}(\Delta_{w^0}+\varepsilon)^{-m_3}\,d\gamma_n\Bigr)^{1/m_3}
\\
&+
32k^3T^2\varepsilon^{-1}
\Bigl(\int_{\mathbb R^n}(\Delta_{w^0}+\varepsilon)^{-m_2}\,d\gamma_n\Bigr)^{1/m_2}
\nonumber
\\
&+
4k^2T
\Bigl(\int_{\mathbb R^n}(\Delta_{w^0}+\varepsilon)^{-m_1}\,d\gamma_n\Bigr)^{1/m_1}.
\nonumber
\end{align*}

\medskip
\noindent
\textbf{Step 3: estimate of the second sum in \eqref{eq-tv-decomp}.}
Let
\[
\Phi_i(y_1,\ldots,y_k):=
\int_{-\infty}^{y_i}\varphi(y_1,\ldots,y_{i-1},t,y_{i+1},\ldots,y_k)\,dt.
\]
Then
$\partial_i\Phi_i=\varphi$.
Since 
\[
w_j^{i-1}=w_j^i \quad \text{for } j\ne i,
\quad
w_i^{i-1}=\widetilde{f}_i,
\quad
w_i^i=\widetilde{g}_i,
\]
we have
\begin{align*}
\nabla\bigl(\Phi_i(w^{i-1})-\Phi_i(w^i)\bigr)
&=
\sum_{j=1}^k
\bigl(\partial_j\Phi_i(w^{i-1})-\partial_j\Phi_i(w^i)\bigr)\nabla w_j^{i-1}
+
\varphi(w^i)\bigl(\nabla \widetilde{f}_i-\nabla \widetilde{g}_i\bigr).
\end{align*}
By the identity
\[
A_{w^{i-1}}M_{w^{i-1}}=\Delta_{w^{i-1}}I,
\]
we obtain
\begin{align*}
\Delta_{w^{i-1}}\bigl(\varphi(w^{i-1})-\varphi(w^i)\bigr)
&=\Delta_{w^{i-1}}\bigl(\partial_i\Phi_i(w^{i-1})-\partial_i\Phi_i(w^i)\bigr)
\\
&=
\sum_{j=1}^k
\bigl\langle
\nabla \Phi_i(w^{i-1})-\nabla \Phi_i(w^i),
\nabla w_j^{i-1}
\bigr\rangle
(A_{w^{i-1}})_{i,j}
\\
&-
\varphi(w^i)
\sum_{j=1}^k
\bigl\langle
\nabla \widetilde{f}_i-\nabla \widetilde{g}_i,
\nabla w_j^{i-1}
\bigr\rangle
(A_{w^{i-1}})_{i,j}.
\end{align*}
Therefore,
\begin{align}\label{eq-second-sum-1}
\int_{\mathbb{R}^n}
(\varphi(w^{i-1})-&\varphi(w^i))
\frac{\Delta_{w^{i-1}}}{\Delta_{w^{i-1}}+\varepsilon}\,d\gamma_n
\\
&=
\int_{\mathbb{R}^n}
(\Delta_{w^{i-1}}+\varepsilon)^{-1}
\sum_{j=1}^k
\bigl\langle
\nabla \Phi_i(w^{i-1})-\nabla \Phi_i(w^i),
\nabla w_j^{i-1}
\bigr\rangle
(A_{w^{i-1}})_{i,j}\,d\gamma_n
\nonumber
\\
&-
\int_{\mathbb{R}^n}
\varphi(w^i)(\Delta_{w^{i-1}}+\varepsilon)^{-1}
\sum_{j=1}^k
\bigl\langle
\nabla \widetilde{f}_i-\nabla \widetilde{g}_i,
\nabla w_j^{i-1}
\bigr\rangle
(A_{w^{i-1}})_{i,j}\,d\gamma_n.
\nonumber
\end{align}

\medskip
\noindent
\textbf{Step 3.1: estimate of the second integral in \eqref{eq-second-sum-1}.}
By Lemma \ref{lem1.1},
\begin{align*}
\Bigl|\sum_{j=1}^k
\bigl\langle
\nabla \widetilde{f}_i-\nabla \widetilde{g}_i,
\nabla w_j^{i-1}
\bigr\rangle
(A_{w^{i-1}})_{i,j}\Bigr|
&\le
\|A_{w^{i-1}} J_{w^{i-1}}\|_{\rm op} |\nabla \widetilde{f}_i-\nabla \widetilde{g}_i|
\\
&\le
2k^{\frac{1}{2}-\frac{k}{2}}\Delta_{w^{i-1}}^{1/2}
\Bigl(\sum_{j=1}^k|\nabla w^{i-1}_j|^2\Bigr)^{\frac{k-1}{2}}
|\nabla \widetilde{f}_i-\nabla \widetilde{g}_i|.
\end{align*}
By the first part of Lemma \ref{lem6.0},
\[
\Delta_{w^{i-1}}^{1/2}
\le
\prod_{m=1}^k|\nabla w_m^{i-1}|,
\]
and therefore
\[
\Bigl\|\Delta_{w^{i-1}}^{1/2}
\Bigl(\sum_{j=1}^k|\nabla w^{i-1}_j|^2\Bigr)^{\frac{k-1}{2}}\Bigr\|_{L^{\frac{p}{2k-1}}(\gamma_n)}
\le
\prod_{m=1}^k\|\nabla w_m^{i-1}\|_{L^p(\gamma_n)}
\Bigl\|\Bigl(\sum_{j=1}^k|\nabla w^{i-1}_j|^2\Bigr)^{\frac{k-1}{2}}\Bigr\|_{L^{\frac{p}{k-1}}(\gamma_n)}.
\]
Since
\[
\Bigl\|\Bigl(\sum_{j=1}^k|\nabla w^{i-1}_j|^2\Bigr)^{\frac{k-1}{2}}\Bigr\|_{L^{\frac{p}{k-1}}(\gamma_n)}
=
\Bigl\|\sum_{j=1}^k|\nabla w^{i-1}_j|^2\Bigr\|_{L^{\frac{p}{2}}(\gamma_n)}^{\frac{k-1}{2}}
\le 
k^{\frac{k-1}{2}} \max_{1\le j\le k}\|\nabla w^{i-1}_j\|_{L^p(\gamma_n)}^{k-1},
\]
we obtain
\[
\Bigl\|\Delta_{w^{i-1}}^{1/2}
\Bigl(\sum_{j=1}^k|\nabla w^{i-1}_j|^2\Bigr)^{\frac{k-1}{2}}\Bigr\|_{L^{\frac{p}{2k-1}}(\gamma_n)}
\le
k^{\frac{k-1}{2}} \max_{1\le j\le k}\|\nabla w^{i-1}_j\|_{L^p(\gamma_n)}^{2k-1}.
\]
Applying Lemma \ref{lem1.11} to the pair $(w^0,w^{i-1})$ with
\[
u=
\Delta_{w^{i-1}}^{1/2}
\Bigl(\sum_{j=1}^k|\nabla w^{i-1}_j|^2\Bigr)^{\frac{k-1}{2}}
\quad\text{and}\quad
v=
|\nabla \widetilde{f}_i-\nabla \widetilde{g}_i|,
\]
we obtain
\begin{align*}
&\biggl|\int_{\mathbb{R}^n}
\varphi(w^i)(\Delta_{w^{i-1}}+\varepsilon)^{-1}
\sum_{j=1}^k
\bigl\langle
\nabla \widetilde{f}_i-\nabla \widetilde{g}_i,
\nabla w_j^{i-1}
\bigr\rangle
(A_{w^{i-1}})_{i,j}\,d\gamma_n\biggr|
\\
&\le
2T\Bigl(\int_{\mathbb R^n}(\Delta_{w^0}+\varepsilon)^{-m_1}\,d\gamma_n\Bigr)^{1/m_1}
+
8kT^2\varepsilon^{-1}
\Bigl(\int_{\mathbb R^n}(\Delta_{w^0}+\varepsilon)^{-m_2}\,d\gamma_n\Bigr)^{1/m_2}.
\end{align*}

\medskip
\noindent
\textbf{Step 3.2: estimate of the first integral in \eqref{eq-second-sum-1}.}
For the first integral in \eqref{eq-second-sum-1}, we integrate by parts:
\begin{align}\label{eq-int-parts-term}
&\int_{\mathbb R^n}
(\Delta_{w^{i-1}}+\varepsilon)^{-1}
\sum_{j=1}^k
\bigl\langle
\nabla \Phi_i(w^{i-1})-\nabla \Phi_i(w^i),
\nabla w_j^{i-1}
\bigr\rangle
(A_{w^{i-1}})_{i,j}\,d\gamma_n
\\
&=
-\int_{\mathbb R^n}
(\Phi_i(w^{i-1})-\Phi_i(w^i))
\sum_{j=1}^k
\frac{(A_{w^{i-1}})_{i,j}Lw_j^{i-1}
+\langle \nabla w_j^{i-1},\nabla (A_{w^{i-1}})_{i,j}\rangle}
{\Delta_{w^{i-1}}+\varepsilon}\,d\gamma_n
\nonumber
\\
&+
\int_{\mathbb R^n}
(\Phi_i(w^{i-1})-\Phi_i(w^i))
\sum_{j=1}^k
\frac{(A_{w^{i-1}})_{i,j}\langle \nabla w_j^{i-1},\nabla \Delta_{w^{i-1}}\rangle}
{(\Delta_{w^{i-1}}+\varepsilon)^2}\,d\gamma_n .
\nonumber
\end{align}
For the second term on the right-hand side, by Lemmas \ref{lem1.1} and \ref{lem1.2} we obtain
\begin{align*}
\Bigl|
\sum_{j=1}^k
(A_{w^{i-1}})_{i,j}\langle \nabla w_j^{i-1},\nabla \Delta_{w^{i-1}}\rangle
\Bigr|
&\le
\|A_{w^{i-1}}J_{w^{i-1}}\|_{\rm op}\,|\nabla \Delta_{w^{i-1}}|
\\
&\le
4k^{\frac32-k}\Delta_{w^{i-1}}
\Bigl(\sum_{j=1}^k |\nabla w_j^{i-1}|^2\Bigr)^{k-1}
\Bigl(\sum_{j=1}^k \|D^2 w_j^{i-1}\|_{\rm HS}^2\Bigr)^{1/2},
\end{align*}
Since
\[
|\Phi_i(w^{i-1})-\Phi_i(w^i)|\le |\widetilde f_i-\widetilde g_i|
\quad\text{and}\quad
\frac{\Delta_{w^{i-1}}}{(\Delta_{w^{i-1}}+\varepsilon)^2}
\le
\frac{1}{\Delta_{w^{i-1}}+\varepsilon},
\]
we obtain
\begin{align*}
&\biggl|
\int_{\mathbb R^n}
(\Phi_i(w^{i-1})-\Phi_i(w^i))
\sum_{j=1}^k
\frac{(A_{w^{i-1}})_{i,j}\langle \nabla w_j^{i-1},\nabla \Delta_{w^{i-1}}\rangle}
{(\Delta_{w^{i-1}}+\varepsilon)^2}\,d\gamma_n
\biggr|
\\
&\le
4k^{\frac32-k}
\int_{\mathbb R^n}
|\widetilde f_i-\widetilde g_i|
\frac{\Bigl(\sum_{j=1}^k |\nabla w_j^{i-1}|^2\Bigr)^{k-1}
\Bigl(\sum_{j=1}^k \|D^2 w_j^{i-1}\|_{\rm HS}^2\Bigr)^{1/2}}
{\Delta_{w^{i-1}}+\varepsilon}\,d\gamma_n.
\end{align*}
By Lemma \ref{lem0}, 
\[
\Bigl\|
\Bigl(\sum_{j=1}^k |\nabla w_j^{i-1}|^2\Bigr)^{k-1}
\Bigl(\sum_{j=1}^k \|D^2 w_j^{i-1}\|_{\rm HS}^2\Bigr)^{1/2}
\Bigr\|_{L^{\frac{p}{2k-1}}(\gamma_n)}
\le
k^{k-\frac12}
\max_{1\le j\le k}\|w_j^{i-1}\|_{\dot W^{2,p}(\gamma_n)}^{2k-1}.
\]
Therefore, by Lemma \ref{lem1.11}, we obtain
\begin{align*}
&\biggl|
\int_{\mathbb R^n}
(\Phi_i(w^{i-1})-\Phi_i(w^i))
\sum_{j=1}^k
\frac{(A_{w^{i-1}})_{i,j}\langle \nabla w_j^{i-1},\nabla \Delta_{w^{i-1}}\rangle}
{(\Delta_{w^{i-1}}+\varepsilon)^2}\,d\gamma_n
\biggr|
\\
&\le
4k\,
T
\Bigl(\int_{\mathbb R^n}(\Delta_{w^0}+\varepsilon)^{-m_1}\,d\gamma_n\Bigr)^{1/m_1}
+
16k^{2}
T^2\varepsilon^{-1}
\Bigl(\int_{\mathbb R^n}(\Delta_{w^0}+\varepsilon)^{-m_2}\,d\gamma_n\Bigr)^{1/m_2}.
\end{align*}

To estimate the first integral on the right-hand side of \eqref{eq-int-parts-term},
we note that, by \eqref{eq-op-norm-est},
\begin{align*}
\Bigl|
\sum_{j=1}^k (A_{w^{i-1}})_{i,j}Lw_j^{i-1}
\Bigr|
&\le
\|A_{w^{i-1}}\|_{\rm op}
\Bigl(\sum_{j=1}^k |Lw_j^{i-1}|^2\Bigr)^{1/2}
\\
&\le
ek^{1-k}
\Bigl(\sum_{j=1}^k |\nabla w_j^{i-1}|^2\Bigr)^{k-1}
\Bigl(\sum_{j=1}^k |Lw_j^{i-1}|^2\Bigr)^{1/2}.
\end{align*}
Therefore, by Lemma \ref{lem0},
\begin{align*}
\Bigl\|
\sum_{j=1}^k (A_{w^{i-1}})_{i,j}Lw_j^{i-1}
\Bigr\|_{L^{\frac{p}{2k-1}}(\gamma_n)}
&\le
ek^{1-k}
\Bigl\|
\Bigl(\sum_{j=1}^k |\nabla w_j^{i-1}|^2\Bigr)^{k-1}
\Bigl(\sum_{j=1}^k |Lw_j^{i-1}|^2\Bigr)^{1/2}
\Bigr\|_{L^{\frac{p}{2k-1}}(\gamma_n)}
\\
&\le
e\sqrt{k}\,
\max_{1\le j\le k}\|w_j^{i-1}\|_{\dot W^{2,p}(\gamma_n)}^{2k-1}
\le
3k\,
\max_{1\le j\le k}\|w_j^{i-1}\|_{\dot W^{2,p}(\gamma_n)}^{2k-1}.
\end{align*}
Next, by Lemma \ref{lem1.2} with $\theta=e_i$,
\[
\Bigl|
\sum_{j=1}^k
\langle \nabla w_j^{i-1},\nabla (A_{w^{i-1}})_{i,j}\rangle
\Bigr|
\le
5k^{\frac32-k}
\Bigl(\sum_{j=1}^k |\nabla w_j^{i-1}|^2\Bigr)^{k-1}
\Bigl(\sum_{j=1}^k \|D^2 w_j^{i-1}\|_{\rm HS}^2\Bigr)^{1/2}.
\]
Hence, by Lemma \ref{lem0},
\begin{align*}
\Bigl\|
\sum_{j=1}^k
\langle \nabla w_j^{i-1},\nabla (A_{w^{i-1}})_{i,j}\rangle
\Bigr\|_{L^{\frac{p}{2k-1}}(\gamma_n)}
&\le
5k^{\frac32-k}
\Bigl\|
\Bigl(\sum_{j=1}^k |\nabla w_j^{i-1}|^2\Bigr)^{k-1}
\Bigl(\sum_{j=1}^k \|D^2 w_j^{i-1}\|_{\rm HS}^2\Bigr)^{1/2}
\Bigr\|_{L^{\frac{p}{2k-1}}(\gamma_n)}
\\
&\le
5k\,
\max_{1\le j\le k}\|w_j^{i-1}\|_{\dot W^{2,p}(\gamma_n)}^{2k-1}.
\end{align*}
Therefore,
\[
\Bigl\|\sum_{j=1}^k(A_{w^{i-1}})_{i,j}Lw_j^{i-1}
+\langle \nabla w_j^{i-1},\nabla (A_{w^{i-1}})_{i,j}\rangle\Bigr\|_{L^{\frac{p}{2k-1}}(\gamma_n)}
\le 8k\max_{1\le j\le k}\max\bigl\{\|f_j\|_{\dot W^{2,p}(\gamma)}^{2k-1}, \|g_j\|_{\dot W^{2,p}(\gamma)}^{2k-1}\bigr\}.
\]
Now we apply Lemma \ref{lem1.11} and obtain
\begin{align*}
&\biggl|
\int_{\mathbb R^n}
(\Phi_i(w^{i-1})-\Phi_i(w^i))
\sum_{j=1}^k
\frac{(A_{w^{i-1}})_{i,j}Lw_j^{i-1}
+\langle \nabla w_j^{i-1},\nabla (A_{w^{i-1}})_{i,j}\rangle}
{\Delta_{w^{i-1}}+\varepsilon}\,d\gamma_n
\biggr|
\\
&\le
8k T
\Bigl(\int_{\mathbb R^n}(\Delta_{w^0}+\varepsilon)^{-m_1}\,d\gamma_n\Bigr)^{1/m_1}
+
32k^2
T^2\varepsilon^{-1}
\Bigl(\int_{\mathbb R^n}(\Delta_{w^0}+\varepsilon)^{-m_2}\,d\gamma_n\Bigr)^{1/m_2}.
\end{align*}

Putting these two estimates together, we obtain
\begin{align*}
&\biggl|
\int_{\mathbb R^n}
(\Delta_{w^{i-1}}+\varepsilon)^{-1}
\sum_{j=1}^k
\bigl\langle
\nabla \Phi_i(w^{i-1})-\nabla \Phi_i(w^i),
\nabla w_j^{i-1}
\bigr\rangle
(A_{w^{i-1}})_{i,j}\,d\gamma_n
\biggr|
\\
&\le
12k\,
T
\Bigl(\int_{\mathbb R^n}(\Delta_{w^0}+\varepsilon)^{-m_1}\,d\gamma_n\Bigr)^{1/m_1}
+
48k^{2}
T^2\varepsilon^{-1}
\Bigl(\int_{\mathbb R^n}(\Delta_{w^0}+\varepsilon)^{-m_2}\,d\gamma_n\Bigr)^{1/m_2}.
\end{align*}

\medskip
\noindent
\textbf{Step 3.3: conclusion of the estimate based on \eqref{eq-second-sum-1}.} 
Combining Steps 3.1 and 3.2 and substituting their estimates into \eqref{eq-second-sum-1}, we obtain
\begin{align*}
&\sum_{i=1}^k
\biggl|
\int_{\mathbb{R}^n}
(\varphi(w^{i-1})-\varphi(w^i))
\frac{\Delta_{w^{i-1}}}{\Delta_{w^{i-1}}+\varepsilon}\,d\gamma_n
\biggr|
\\
&\le
14k^2T\Bigl(\int_{\mathbb R^n}(\Delta_{w^0}+\varepsilon)^{-m_1}\,d\gamma_n\Bigr)^{1/m_1}
+
56k^3T^2\varepsilon^{-1}
\Bigl(\int_{\mathbb R^n}(\Delta_{w^0}+\varepsilon)^{-m_2}\,d\gamma_n\Bigr)^{1/m_2}.
\end{align*}

\medskip
\noindent
\textbf{Step 4: the resulting estimate.}
To obtain the final estimate in the cylindrical case, we use
\[
\frac{1}{\Delta_{w^k}+\varepsilon}
\le
\frac{|\Delta_{w^0}-\Delta_{w^k}|}{\varepsilon(\Delta_{w^0}+\varepsilon)}
+
\frac{1}{\Delta_{w^0}+\varepsilon},
\]
which yields
\begin{align*}
\biggl|
\int_{\mathbb R^n}
\varphi(w^k)\frac{\varepsilon}{\Delta_{w^k}+\varepsilon}\,d\gamma_n
\biggr|
&\le
\int_{\mathbb R^n}
\frac{|\Delta_{w^0}-\Delta_{w^k}|}{\Delta_{w^0}+\varepsilon}\,d\gamma_n
+
\int_{\mathbb R^n}
\frac{\varepsilon}{\Delta_{w^0}+\varepsilon}\,d\gamma_n
\\
&\le
\|\Delta_{w^0}-\Delta_{w^k}\|_{L^{\frac{p}{2k}}(\gamma_n)}
\Bigl(
\int_{\mathbb R^n}
(\Delta_{w^0}+\varepsilon)^{-m_1}\,d\gamma_n
\Bigr)^{1/m_1}
+
\int_{\mathbb R^n}
\frac{\varepsilon}{\Delta_{w^0}+\varepsilon}\,d\gamma_n
\\
&\le
4kT
\Bigl(
\int_{\mathbb R^n}
(\Delta_{w^0}+\varepsilon)^{-m_1}\,d\gamma_n
\Bigr)^{1/m_1}
+
\int_{\mathbb R^n}
\frac{\varepsilon}{\Delta_{w^0}+\varepsilon}\,d\gamma_n.
\end{align*}
Substituting this estimate, together with the estimates obtained in Steps 2 and 3, into \eqref{eq-tv-decomp}, and then passing back to the measure $\gamma$, we obtain
\begin{align}\label{eq-TV-cyl-est}
\biggl|\int_E (\varphi(f)-\varphi(g))\,d\gamma\biggr|
&\le
64k^4T^3\varepsilon^{-2}
\Bigl(\int_E(\Delta_f+\varepsilon)^{-m_3}\,d\gamma\Bigr)^{1/m_3}
\\
&+
88k^3T^2\varepsilon^{-1}
\Bigl(\int_E(\Delta_f+\varepsilon)^{-m_2}\,d\gamma\Bigr)^{1/m_2}
\nonumber
\\
&+
22k^2T
\Bigl(\int_E(\Delta_f+\varepsilon)^{-m_1}\,d\gamma\Bigr)^{1/m_1}
\nonumber
\\
&+
2\int_E \frac{\varepsilon}{\Delta_f+\varepsilon}\,d\gamma.
\nonumber
\end{align}

\medskip
\noindent
\textbf{Step 5: approximation by cylindrical functions.}
Let now
\[
f=(f_1,\ldots,f_k),\quad g=(g_1,\ldots,g_k),
\quad
f_j,g_j\in W^{2,p}(\gamma),
\]
and assume that
\[
\max_{1\le j\le k}\|f_j\|_{\dot W^{2,p}(\gamma)}\le b
\quad\text{and}\quad
\max_{1\le j\le k}\|g_j\|_{\dot W^{2,p}(\gamma)}\le b.
\]
Choose
\[
f_j^n,g_j^n\in \mathcal{FC}^\infty,
\quad j=1,\ldots,k,
\]
such that
\[
f_j^n\to f_j
\quad\text{and}\quad
g_j^n\to g_j
\quad\text{in}\quad
W^{2,p}(\gamma).
\]
Set
\[
f^n:=(f_1^n,\ldots,f_k^n)
\quad\text{and}\quad
g^n:=(g_1^n,\ldots,g_k^n).
\]
Passing to a subsequence if necessary, we may assume that
\[
f_j^n\to f_j,
\quad
g_j^n\to g_j,
\quad
\nabla f_j^n\to \nabla f_j,
\quad
\nabla g_j^n\to \nabla g_j
\quad\gamma\text{-a.e.}
\]
In particular,
\[
\Delta_{f^n}\to \Delta_f
\quad\gamma\text{-a.e.}
\]
Therefore, for each $m\in\{1,m_1,m_2,m_3\}$,
\[
(\Delta_{f^n}+\varepsilon)^{-m}\to (\Delta_f+\varepsilon)^{-m}
\quad\gamma\text{-a.e.}
\]
Since
\[
0\le (\Delta_{f^n}+\varepsilon)^{-m}\le \varepsilon^{-m},
\]
the dominated convergence theorem yields
\[
\int_E(\Delta_{f^n}+\varepsilon)^{-m}\,d\gamma
\to
\int_E(\Delta_f+\varepsilon)^{-m}\,d\gamma
\quad \forall m\in\{1,m_1,m_2,m_3\}.
\]
Since $\varphi$ is bounded and continuous, the dominated convergence theorem also gives
\[
\int_E (\varphi(f^n)-\varphi(g^n))\,d\gamma
\to
\int_E (\varphi(f)-\varphi(g))\,d\gamma.
\]
Finally, we have
\[
T_n\to T,
\]
where
\[
T_n:=
\max_{1\le j\le k}\max\bigl\{\|f_j^n\|_{\dot W^{2,p}(\gamma)}^{2k-1}, \|g_j^n\|_{\dot W^{2,p}(\gamma)}^{2k-1}\bigr\}
\max_{1\le j\le k}\|f_j^n-g_j^n\|_{W^{1,p}(\gamma)}
\]
and
\[
T:=
\max_{1\le j\le k}\max\bigl\{\|f_j\|_{\dot W^{2,p}(\gamma)}^{2k-1}, \|g_j\|_{\dot W^{2,p}(\gamma)}^{2k-1}\bigr\}
\max_{1\le j\le k}\|f_j-g_j\|_{W^{1,p}(\gamma)}\le 
b^{2k-1}\max_{1\le j\le k}\|f_j-g_j\|_{W^{1,p}(\gamma)}.
\]
Passing to the limit in \eqref{eq-TV-cyl-est}, we obtain the same estimate,
now for arbitrary $f_j,g_j\in W^{2,p}(\gamma)$.

\medskip
\noindent
\textbf{Step 6: optimization in $\varepsilon$.}
If $T=0$, then $f=g$ in $W^{1,p}(\gamma)$, and there is nothing to prove. Assume therefore that $T>0$.

Since
\[
\gamma(\Delta_f\le s)\le as^\varkappa
\quad \forall s>0,
\]
for every $m\ge1>\varkappa$ we have
\begin{align*}
\int_E(\Delta_f+\varepsilon)^{-m}\,d\gamma
&=
m\int_0^\infty (s+\varepsilon)^{-m-1}\gamma(\Delta_f\le s)\,ds
\le
am\int_0^\infty (s+\varepsilon)^{-m-1}s^\varkappa\,ds
\\
&\le
am\int_0^\infty (s+\varepsilon)^{-m-1+\varkappa}\,ds
=
a\frac{m}{m-\varkappa}\varepsilon^{\varkappa-m}
\le
(1-\varkappa)^{-1}a\,\varepsilon^{\varkappa-m},
\end{align*}
because $m\ge 1$. 
Hence
\begin{align*}
\Bigl|\int_E (\varphi(f)-\varphi(g))\,d\gamma\Bigr|
&\le
64k^4(1-\varkappa)^{-1}a^{1/m_3}T^3\varepsilon^{-3+\varkappa/m_3}
+
88k^3(1-\varkappa)^{-1}a^{1/m_2}T^2\varepsilon^{-2+\varkappa/m_2}
\\
&+
22k^2(1-\varkappa)^{-1}a^{1/m_1}T\varepsilon^{-1+\varkappa/m_1}
+
2(1-\varkappa)^{-1}a\,\varepsilon^\varkappa.
\end{align*}
Set
\[
\eta:=1+\frac{2k\varkappa}{p}.
\]
Then
\[
-3+\frac{\varkappa}{m_3}=\varkappa-3\eta,
\quad
-2+\frac{\varkappa}{m_2}=\varkappa-2\eta,
\quad
-1+\frac{\varkappa}{m_1}=\varkappa-\eta.
\]
Now choose
\[
\varepsilon:=a^{-\frac{2k}{p\eta}}T^{\frac{1}{\eta}}.
\]
Since
\[
\beta:=\frac{p\varkappa}{p+2k\varkappa}=\frac{\varkappa}{\eta},
\]
a direct computation gives
\[
a^{1/m_3}T^3\varepsilon^{-3+\varkappa/m_3}
=
a^{1/m_2}T^2\varepsilon^{-2+\varkappa/m_2}
=
a^{1/m_1}T\varepsilon^{-1+\varkappa/m_1}
=
a\varepsilon^\varkappa
=
a^{\beta/\varkappa}T^\beta.
\]
Therefore,
\begin{align*}
\biggl|\int_E (\varphi(f)-\varphi(g))\,d\gamma\biggr|
&\le
(64k^4+88k^3+22k^2+2)(1-\varkappa)^{-1}a^{\beta/\varkappa}T^\beta
\\
&\le
176k^4(1-\varkappa)^{-1}a^{\beta/\varkappa}T^\beta.
\end{align*}
Since
\[
T\le b^{2k-1}\max_{1\le j\le k}\|f_j-g_j\|_{W^{1,p}(\gamma)},
\]
taking the supremum over all $\varphi\in C_0^\infty(\mathbb R^k)$ with $\|\varphi\|_\infty\le 1$, we get
\[
d_{\rm TV}(f,g)
\le
176k^4(1-\varkappa)^{-1}a^{\beta/\varkappa}
\bigl(b^{2k-1}\max_{1\le j\le k}\|f_j-g_j\|_{W^{1,p}(\gamma)}\bigr)^\beta.
\]
This completes the proof.
\end{proof}

\begin{theorem}\label{t3.2}
Let $k,d\in\mathbb N$ with $d\ge 2$, and let $a,b>0$. Set
\[
\varkappa:=\frac{1}{2k(d-1)}
\quad\text{and}\quad
s:=\frac{d}{2(d-1)}.
\]
Then, for every pair of mappings
\[
f=(f_1,\ldots,f_k),
\quad
g=(g_1,\ldots,g_k),
\quad
f_j,g_j\in \mathcal P_d(\gamma),
\]
satisfying
\[
\int_E \Delta_f\,d\gamma \ge a
\quad
\text{and}
\quad
\max_{1\le j\le k}\operatorname{Var}_\gamma(f_j)\le b,
\]
and for every $\beta\in(0,\varkappa)$, one has
\begin{equation}\label{eq-L2-est-1}
d_{\rm TV}(f,g)
\le
C_1k^5d(\varkappa-\beta)^{-s}
\bigl(a^{-1}b^{\frac{2k-1}{2}}\max_{1\le j\le k}\|f_j-g_j\|_{L^2(\gamma)}\bigr)^\beta,
\end{equation}
where $C_1>0$ is an absolute constant.
	
In particular, optimizing over $\beta$, we obtain
\[
d_{\rm TV}(f,g)
\le
C_2k^5d
\bigl(1+\bigl|\ln\bigl(a^{-1}b^{\frac{2k-1}{2}}\max_{1\le j\le k}\|f_j-g_j\|_{L^2(\gamma)}\bigr)\bigr|\bigr)^s
\bigl(a^{-1}b^{\frac{2k-1}{2}}\max_{1\le j\le k}\|f_j-g_j\|_{L^2(\gamma)}\bigr)^\varkappa
\]
for some absolute constant $C_2>0$.
\end{theorem}

\begin{proof}
The proof follows the same lines as the proof of Theorem \ref{t1.2}.
Fix $p>6k$, which will be specified later.
Since $d_{\rm TV}(f,g)$ and 
$\max_{1\le j\le k}\|f_j-g_j\|_{L^2(\gamma)}$
are invariant under the shift of both mappings by the same constant vector,
without loss of generality we may assume that
\[
\int_Ef_j\, d\gamma=0\quad \forall j\in\{1,\ldots, k\}.
\]	
Set
\[
D:=\max_{1\le j\le k}\|f_j-g_j\|_{L^2(\gamma)}.
\]

By \eqref{eq-Sobolev-est} we have
\[
\|f_j\|_{\dot W^{2,p}(\gamma)}
\le
3d(p-1)^{d/2}\|f_j\|_{L^2(\gamma)}\le 3d(p-1)^{d/2}\sqrt b
\]
and
\[
\|f_j-g_j\|_{\dot W^{2,p}(\gamma)}
\le 3d(p-1)^{d/2}\|f_j-g_j\|_{L^2(\gamma)}
\le 3d(p-1)^{d/2}D
\]
implying that
\[
\|g_j\|_{\dot W^{2,p}(\gamma)}\le
3d(p-1)^{d/2}\bigl(D + \sqrt{b}\bigr).
\]
By hypercontractivity \eqref{eq-hyper},
\[
\|f_j-g_j\|_{L^p(\gamma)}\le (p-1)^{d/2}\|f_j-g_j\|_{L^2(\gamma)}.
\]
By \eqref{eq-D-1},
\[
\|\nabla f_j - \nabla g_j\|_{L^p(\gamma)}
\le
\sqrt{d}(p-1)^{(d-1)/2}\|f_j-g_j\|_{L^2(\gamma)}.
\]
Therefore,
\[
\|f_j-g_j\|_{W^{1,p}(\gamma)}
\le
2\sqrt d\,(p-1)^{d/2}\|f_j-g_j\|_{L^2(\gamma)},
\]
and hence
\[
\max_{1\le j\le k}\|f_j-g_j\|_{W^{1,p}(\gamma)}
\le
2\sqrt d\,(p-1)^{d/2}D.
\]
Since $\Delta_f\in \mathcal P_{2k(d-1)}(\gamma)$,
by the Carbery--Wright inequality (see \cite{CW01}),
\[
\gamma(\Delta_f\le \varepsilon)
\le
Ckd\,a^{-\frac{1}{2k(d-1)}}\varepsilon^{\frac{1}{2k(d-1)}}
=
Ckda^{-\varkappa}\varepsilon^\varkappa
\]
for all $\varepsilon>0$, where $C\ge1$ is an absolute constant.
	
Assume first that $D\le \sqrt{b}$ and apply Theorem \ref{t3.1} with
\[
\beta_p:=\frac{p\varkappa}{p+2k\varkappa}.
\]
We obtain
\begin{align*}
d_{\rm TV}(f,g)
&\le
176k^4(1-\varkappa)^{-1}
(Ckd)^{\beta_p/\varkappa}
a^{-\beta_p}
\bigl(6d(p-1)^{d/2}\sqrt b\bigr)^{(2k-1)\beta_p}
\bigl(2\sqrt d\,(p-1)^{d/2}D\bigr)^{\beta_p}.
\end{align*}
Since 
\[
(1-\varkappa)^{-1}\le 2
\quad\text{and}\quad
\beta_p\le \varkappa=\frac{1}{2k(d-1)}\le \frac{1}{2k},
\]
we have
\[
(Ckd)^{\beta_p/\varkappa}\le Ckd
\]
and
\[
(6d)^{(2k-1)\beta_p}(2\sqrt d)^{\beta_p}
=
6^{(2k-1)\beta_p}2^{\beta_p}d^{(2k-\frac12)\beta_p}
\le 12\, d^{\frac{1}{d-1}}\le 24.
\]
Furthermore,
\[
\frac d2(2k-1)\beta_p+\frac d2\beta_p=kd\,\beta_p\le kd\,\varkappa=s.
\]
Thus,
\[
d_{\rm TV}(f,g)
\le
C_0k^5d\,p^s
\bigl(a^{-1}b^{\frac{2k-1}{2}}D\bigr)^{\beta_p},
\]
where $C_0\ge4$ is an absolute constant.
	
If $D\ge \sqrt b$, then, by Lemma \ref{lem6.0} and \eqref{eq-D-1},
\[
a\le \int_E\Delta_f\, d\gamma
\le
\int_E\prod_{j=1}^k|\nabla f_j|^2\, d\gamma
\le
\prod_{j=1}^k\|\nabla f_j\|_{L^p(\gamma)}^2
\le
d^k(p-1)^{k(d-1)}b^k.
\]
Therefore,
\[
d_{\rm TV}(f,g)
\le
2\le 2(b^{-1/2}D)^{\beta_p}
\le
2d^{k\beta_p}(p-1)^{k(d-1)\beta_p}\bigl(a^{-1}b^{\frac{2k-1}{2}}D\bigr)^{\beta_p}.
\]
Moreover,
\[
d^{k\beta_p}
\le
d^{k\varkappa}
=
d^{\frac{1}{2(d-1)}}\le 2
\]
and
\[
(p-1)^{k(d-1)\beta_p}
\le
p^{kd\varkappa}
= p^s.
\]
Thus, in the case $D\ge \sqrt b$, we also obtain
\[
d_{\rm TV}(f,g)
\le
4p^s\bigl(a^{-1}b^{\frac{2k-1}{2}}D\bigr)^{\beta_p}
\le
C_0k^5d\,p^s\bigl(a^{-1}b^{\frac{2k-1}{2}}D\bigr)^{\beta_p}.
\]

Now fix $\beta\in(0,\varkappa)$ and choose
\[
p:=(\varkappa-\beta)^{-1}+6k.
\]
Then $p>6k$. Moreover,
\[
\varkappa-\beta_p
=
\frac{2k\varkappa^2}{p+2k\varkappa}
\le
\frac{1}{p}\le \varkappa - \beta,
\]
that is, $\beta\le \beta_p$.
We also note that
\[
(\varkappa-\beta)^{-1}>\varkappa^{-1}=2k(d-1)\ge 2k,
\]
so
\[
6k\le 3(\varkappa-\beta)^{-1}
\quad\text{and}\quad
p\le 4(\varkappa-\beta)^{-1},
\]
and thus
\[
p^s\le 4^s(\varkappa-\beta)^{-s}\le 4(\varkappa-\beta)^{-s}.
\]
	
Set
\[
A:=a^{-1}b^{\frac{2k-1}{2}}D.
\]
If $A\le 1$, then
$A^{\beta_p}\le A^\beta$
and
\[
d_{\rm TV}(f,g)
\le
4C_0k^5d(\varkappa-\beta)^{-s}A^\beta.
\]
If $A\ge 1$, then
\[
d_{\rm TV}(f,g)\le 2\le 4C_0k^5d(\varkappa-\beta)^{-s}A^\beta.
\]
This proves \eqref{eq-L2-est-1} with $C_1=4C_0$.
	
To optimize over $\beta$, we argue exactly as in the proof of Theorem \ref{t1.2}. If
$A<e^{-1/\varkappa}$,
choose
\[
\beta=\varkappa-\frac{1}{\ln A^{-1}}.
\]
Then $\beta\in(0,\varkappa)$
and
\[
(\varkappa-\beta)^{-s}A^\beta
=
(\ln A^{-1})^sA^{\varkappa-\frac{1}{\ln A^{-1}}}
=
e(\ln A^{-1})^sA^\varkappa.
\]
If
$A\ge e^{-1/\varkappa}$,
then
\[
d_{\rm TV}(f,g)\le 2\le 2e\,A^\varkappa
\le
2ek^5d\bigl(1+|\ln A|\bigr)^sA^\varkappa.
\]
Consequently, in both cases,
\[
d_{\rm TV}(f,g)
\le
C_2k^5d
\bigl(1+\bigl|\ln\bigl(a^{-1}b^{\frac{2k-1}{2}}D\bigr)\bigr|\bigr)^s
\bigl(a^{-1}b^{\frac{2k-1}{2}}D\bigr)^\varkappa,
\]
which completes the proof.
\end{proof}

\section*{Use of AI Tools}

ChatGPT was used for language editing, stylistic suggestions, draft wording for
selected passages, and help with locating some references. All AI-generated
text and suggested references were checked, corrected where necessary, and
substantially revised by the authors. The authors take full responsibility for
the content of the paper.





\section*{Acknowledgements}

This work was supported by the AEI grants 
RYC2023-043616-I and
PID2025-169712NA-I00 funded by MICIU/AEI/10.13039/501100011033,
and by the Spanish State Research Agency, through the Severo Ochoa and Mar\'ia de Maeztu Program for Centers and
Units of Excellence in R\&D (CEX2020-001084-M).
The authors thanks CERCA Programme (Generalitat de Catalunya) for institutional support.
{\sloppy
	
}


\begin{thebibliography}{99}
\normalsize
	
\bibitem{BC14}
Bally, V., Caramellino, L.:
On the distances between probability density functions.
Electron. J. Probab. \textbf{19}, Paper No.~110, 1--33 (2014)
	
\bibitem{BC17}
Bally, V., Caramellino, L.:
Convergence and regularity of probability laws by using an interpolation method.
Ann. Probab. \textbf{45}(2), 1110--1159 (2017)
	
\bibitem{BC19}
Bally, V., Caramellino, L.:
Total variation distance between stochastic polynomials and invariance principles.
Ann. Probab. \textbf{47}(6), 3762--3811 (2019)
	
\bibitem{BCP}
Bally, V., Caramellino, L., Poly, G.:
Regularization lemmas and convergence in total variation.
Electron. J. Probab. \textbf{25}, Paper No.~92, 1--20 (2020)
	
\bibitem{BIN}
Besov, O.V., Il'in, V.P., Nikol'skii, S.M.:
\emph{Integral Representations of Functions and Imbedding Theorems}, Vols.~I, II.
Winston, Washington; Halsted Press, New York--Toronto--London (1978, 1979)
	
\bibitem{Gaus}
Bogachev, V.I.:
\emph{Gaussian Measures}.
Mathematical Surveys and Monographs, vol.~62. American Mathematical Society, Providence (1998)
	
\bibitem{BKZ}
Bogachev, V.I., Kosov, E.D., Zelenov, G.I.:
Fractional smoothness of distributions of polynomials and a fractional analog of the Hardy--Landau--Littlewood inequality.
Trans. Amer. Math. Soc. \textbf{370}(6), 4401--4432 (2018)
	
\bibitem{CW01}
Carbery, A., Wright, J.:
Distributional and $L^q$ norm inequalities for polynomials over convex bodies in $\mathbb R^n$.
Math. Res. Lett. \textbf{8}(3), 233--248 (2001)
	
\bibitem{DM87}
Davydov, Y.A., Martynova, G.V.:
Limit behavior of multiple stochastic integrals.
In: \emph{Statistics and Control of Random Processes}, pp.~55--57.
Nauka, Moscow (1987)
	
\bibitem{ENP25}
Ebina, M., Nourdin, I., Peccati, G.:
Optimal local central limit theorems on Wiener chaos.
Preprint, arXiv:2511.21496 (2025)
	
\bibitem{HMP25}
Herry, R., Malicet, D., Poly, G.:
Regularity of laws via Dirichlet forms: application to quadratic forms in independent and identically distributed random variables.
Probab. Theory Related Fields \textbf{191}, 523--567 (2025)
	
\bibitem{HMP24}
Herry, R., Malicet, D., Poly, G.:
Superconvergence phenomenon in Wiener chaoses.
Ann. Probab. \textbf{52}(3), 1162--1200 (2024)
	
\bibitem{HMP}
Herry, R., Malicet, D., Poly, G.:
Limit distributions for polynomials with independent and identically distributed entries.
Preprint, arXiv:2412.06749 (2024)
	
\bibitem{HLN14}
Hu, Y., Lu, F., Nualart, D.:
Convergence of densities of some functionals of Gaussian processes.
J. Funct. Anal. \textbf{266}(2), 814--875 (2014)
	
\bibitem{Jan97}
Janson, S.:
\emph{Gaussian Hilbert Spaces}.
Cambridge Tracts in Mathematics, vol.~129. Cambridge University Press, Cambridge (1997)
	
\bibitem{Kos}
Kosov, E.D.:
Fractional smoothness of images of logarithmically concave measures under polynomials.
J. Math. Anal. Appl. \textbf{462}(1), 390--406 (2018)
	
\bibitem{Kos-FCAA}
Kosov, E.D.:
On fractional regularity of distributions of functions in Gaussian random variables.
Fract. Calc. Appl. Anal. \textbf{22}(5), 1249--1268 (2019)
	
\bibitem{Kos-MS}
Kosov, E.D.:
Besov classes on finite and infinite dimensional spaces.
Sb. Math. \textbf{210}(5), 663--692 (2019)
	
\bibitem{Kos-IMRN}
Kosov, E.D.:
Total variation distance estimates via $L^2$-norm for polynomials in log-concave random vectors.
Int. Math. Res. Not. IMRN \textbf{2021}(21), 16492--16508 (2021)
	
\bibitem{Kos-Adv}
Kosov, E.D.:
Regularity of linear and polynomial images of Skorohod differentiable measures.
Adv. Math. \textbf{397}, Paper No.~108193 (2022)
	
\bibitem{Kos23}
Kosov, E.D.:
Regularity of distributions of Sobolev mappings in abstract settings.
Math. Notes \textbf{114}(5), 862--874 (2023)
	
\bibitem{KosZh}
Kosov, E.D., Zhukova, A.K.:
Improved bounds for the total variation distance between stochastic polynomials.
Stochastic Process. Appl. \textbf{170}, Paper No.~104279 (2024)
	
\bibitem{Kos25}
Kosov, E.D.:
Oscillatory integrals with polynomial phase and regularity of distributions.
Preprint, arXiv:2511.02679 (2025)

\bibitem{NNP}
Nourdin, I., Nualart, D., Poly, G.:
Absolute continuity and convergence of densities for random vectors on Wiener chaos.
Electron. J. Probab. \textbf{18}, Paper No.~22, 1--19 (2013)
	
\bibitem{NN16}
Nourdin, I., Nualart, D.:
Fisher information and the fourth moment theorem.
Ann. Inst. Henri Poincar\'e Probab. Stat. \textbf{52}(2), 849--867 (2016)
	
\bibitem{NP09}
Nourdin, I., Peccati, G.:
Stein's method on Wiener chaos.
Probab. Theory Related Fields \textbf{145}(1--2), 75--118 (2009)
	
\bibitem{NP12}
Nourdin, I., Peccati, G.:
\emph{Normal Approximations with Malliavin Calculus: From Stein's Method to Universality}.
Cambridge Tracts in Mathematics, vol.~192. Cambridge University Press, Cambridge (2012)
	
\bibitem{NP13}
Nourdin, I., Poly, G.:
Convergence in total variation on Wiener chaos.
Stochastic Process. Appl. \textbf{123}(2), 651--674 (2013)
	
\bibitem{NP05}
Nualart, D., Peccati, G.:
Central limit theorems for sequences of multiple stochastic integrals.
Ann. Probab. \textbf{33}(1), 177--193 (2005)
	
\bibitem{PT04}
Peccati, G., Tudor, C.A.:
Gaussian limits for vector-valued multiple stochastic integrals.
In: \emph{S\'eminaire de Probabilit\'es XXXVIII}, Lecture Notes in Math., vol.~1857, pp.~247--262.
Springer, Berlin (2005)
	
\bibitem{Stein}
Stein, E.M.:
\emph{Singular Integrals and Differentiability Properties of Functions}.
Princeton Mathematical Series, vol.~30. Princeton University Press, Princeton (1970)
	
\end{thebibliography}
\end{document}